\documentclass{amsart}
\usepackage{amsmath,amsfonts,amsthm,amssymb,framed}
\usepackage[pdftex]{graphicx}
\usepackage[pdftex,hyperfootnotes=false]{hyperref}

\usepackage{tikz}

\usetikzlibrary{arrows,decorations.pathmorphing,backgrounds,positioning,fit,trees}
\usepackage{setspace}
\newtheorem{thm}{Theorem}[section]
\newtheorem*{thm*}{Theorem}
\newtheorem*{lemma*}{Lemma}
\newtheorem*{coro*}{Corollary}
\newtheorem{lemma}[thm]{Lemma}

\newtheorem{coro}[thm]{Corollary}
\newtheorem{prop}[thm]{Proposition}
\theoremstyle{definition}
\newtheorem{defn}[thm]{Definition}
\theoremstyle{remark}
\newtheorem{rem}[thm]{Remark}
\newtheorem{ex}[thm]{Example}
\newtheorem{quest}[thm]{Question}
\numberwithin{equation}{section}
\def\k{\mathsf{k}}
\def\A{\mathcal{A}}
\def\U{\mathcal{U}}
\def\X{\mathcal{X}}
\def\S{\Sigma}

\def\R{\mathbb{R}}

\def\a{\mathbf{a}}

\def\Z{\mathbb{Z}}
\def\Q{\mathsf{Q}}
\def\M{\mathbb{M}}
\def\J{\mathcal{J}}

\title{Locally acyclic cluster algebras.}

\author{Greg Muller}
\address{Department of Mathematics,
Louisiana State University, Baton Rouge, LA 70808, USA}
\email{gmuller@lsu.edu}
\thanks{This work was supported by the VIGRE program at LSU, National Science Foundation grant DMS-0739382.}

\tikzstyle{mutable}=[inner sep=0.5mm,circle,draw,minimum size=2mm]
\tikzstyle{frozen}=[inner sep=0.5mm,rectangle,draw]
\tikzstyle{marked}=[inner sep=0.5mm,circle,draw,fill=black!50]

\begin{document}

\begin{abstract}
This paper studies cluster algebras locally, by identifying a special class of localizations which are themselves cluster algebras.  A `locally acyclic cluster algebra' is a cluster algebra which admits a finite cover (in a geometric sense) by acyclic cluster algebras.  Many important results about acyclic cluster algebras extend to local acyclic cluster algebras (such as finite generation, integrally closure, and equaling their upper cluster algebra), as well as results which are new even for acyclic cluster algebras (such as regularity when the exchange matrix has full rank).

We develop several techniques for determining whether a cluster algebra is locally acyclic.  We show that cluster algebras of marked surfaces with at least two boundary marked points are locally acyclic, providing a large class of examples of cluster algebras which are locally acyclic but not acyclic.  We also work out several specific examples in detail.
\end{abstract}

\maketitle

\section{Introduction}

\subsection{Cluster algebras}

Cluster algebras were introduced by Fomin and Zelevinsky \cite{FZ02} in the study of canonical bases of Lie algebras.  Since then, cluster algebras have been discovered in many contexts throughout mathematics.  They often arise as the coordinate rings of spaces (such as double Bruhat cells \cite{BFZ05}, Grassmannians \cite{GLS08}, or decorated Teichm\"uller space \cite{GSV05}) where certain special functions on that space (such as generalized minors, Pl\"ucker coordinates or Penner coordinates) correspond to special elements called \emph{cluster variables}.

The distinguishing feature of a cluster algebra is that the full set of cluster variables can be recovered from special $n$-element subsets of the cluster variables called \emph{clusters}.  A cluster variable in a cluster can be \emph{mutated} to produce a new cluster variable.  Iterating in all possible ways reconstructs the complete set of cluster variables.  This generalizes known relations like the Desnanot-Jacobi identity, the three-term Pl\"ucker relation, and the Penner relation.

This paper will focus on \emph{skew-symmetric} cluster algebras of geometric type.  This are less than completely general, but they can be encoded into an \emph{ice quiver}; a quiver with some vertices declared as \emph{frozen}.  This turns a potentially complicated algebraic object into a simpler and more visual combinatorial object.  This is an expository choice; \textbf{all the results in this paper remain true for skew-symmetrizable cluster algebras}.\footnote{However, if a reader feels that the skew-symmetrizable versions of the proofs merit publication, they will receive no contention or competition from this author.}


\subsection{Acyclic cluster algebras} There is a curious dichotomy in the general theory of cluster algebras.  Cluster algebras seem to either be very well-behaved or very poorly-behaved, with few examples in between.  By way of illustration, a cluster algebra $\A$ of rank 3 is either generated by 6 elements or is non-Noetherian.\footnote{Either $\A$ is acyclic, and generated by six elements by Lemma \ref{lemma: acyclicpres}, or it behaves similarly to the Markov cluster algebra (see Section \ref{section: Markov}).}  One possible explanation for this is that a related algebra, called the \emph{upper cluster algebra} $\U$, is the truly well-behaved object.  However, the upper cluster algebra is defined as an infinite intersection of algebras, and so it can be hard to work with directly.

In \cite{BFZ05}, Berenstein, Fomin and Zelevinsky defined \emph{acyclic} cluster algebras.  These are cluster algebras which can come from a quiver with no directed cycle of unfrozen vertices.  They produce an explicit finite presentation of any acyclic cluster algebra, and show that it coincides with its upper cluster algebra.  Since then, many results on cluster algebras have focused on the acyclic case.

\subsection{Locally acyclic cluster algebras}

Despite the success of working with acyclic cluster algebras, there are still many examples of cluster algebras which are not acyclic but which enjoy many of the same properties. We will try generalize acyclic cluster algebras to capture more of these well-behaved algebras.

To do this, we consider localizations of cluster algebras where a finite set of cluster variables has been inverted.  Under additional conditions (Lemma \ref{lemma: acycliclocal}), this localization will again be a cluster algebra in a natural way; call this a \emph{cluster localization}.  On the level of ice quivers, a cluster localization will amount to freezing some vertices, and so the cluster localization will be a simpler cluster algebra.

Geometrically, a localization of a commutative algebra $\A$ defines an open subscheme in the affine scheme $Spec(\A)$.  Therefore, a cluster localization $\A'$ of $\A$ defines an open patch in $\X_\A:=Spec(\A)$, called a \emph{cluster chart}.  If $\X_\A$ can be covered by smaller cluster charts, then any local property of $\X_\A$ or $\A$ can checked on the cover.\footnote{By a local property, we mean any property of a scheme or an algebra which is true if and only if it is true on any finite open cover.}

Since many properties are already known about acyclic cluster algebras, this motivates the following definition.
\begin{defn}
A \textbf{locally acyclic cluster algebra} $\A$ is a cluster algebra such that $\X_\A$ has a finite cover by acyclic cluster charts.
\end{defn}

\subsection{Properties of locally acyclic cluster algebras} It is immediate that any local property of acyclic cluster algebras is also true of locally acyclic cluster algebras.  We list these now.

Let $\U$ denote the upper cluster algebra of $\A$.
\begin{thm*} [\textbf{\ref{thm: A=U}}]
If $\A$ is locally acyclic, then $\A=\U$.
\end{thm*}
Recall that a commutative domain is \emph{integrally closed} if it contains the root of any monic polynomial which exists in its fraction field; and a commutative ring is \emph{locally a complete intersection} if it is locally a quotient of a regular ring by a regular sequence.\footnote{A variety is locally a complete intersections if it can be locally defined in affine space using the minimal number of equations, given its dimension.  It also implies other geometric properties, such as being Gorenstein and Cohen-Macaulay.}
\begin{thm*} [\textbf{\ref{thm: LAprops}}]
If $\A$ is locally acyclic, then $\A$ is finitely-generated, integrally closed, and locally a complete intersection.
\end{thm*}
In \cite{GSV05}, the authors define a 2-form $\omega_{WP}$ on an open subset of $\X_\A$ called the \emph{Weil-Petersson 2-form}.
\begin{thm*} [\textbf{\ref{thm: WP}}]
If $\A$ is locally acyclic, the Weil-Petersson form $\omega_{WP}$ extends to a K\"ahler differential 2-form on $\X_\A$.
\end{thm*}

Given an ice quiver $\Q$, the \emph{exchange matrix} of $\Q$ is a matrix with a column for each vertex and a row for each mutable vertex; in each entry, it counts the signed number of arrows between vertices.  It is known that the rank of this matrix is an invariant of the cluster algebra \cite{GSV03}.  A cluster algebra is \emph{full-rank} if any exchange matrix has full rank.
\begin{thm*} [\textbf{\ref{thm: LAsmooth}}]
If $\A$ is a locally acyclic cluster algebra of full rank, then $\mathbb{Q}\otimes\A$ is regular.
\end{thm*}
\begin{coro*}[\textbf{\ref{coro: LAsmooth}}]
If $\A$ is a locally acyclic cluster algebra of full rank, then $\X_\A(\mathbb{C}):=Hom(\A,\mathbb{C})$ is a smooth complex manifold, and $\X_\A(\R):=Hom(\A,\mathbb{R})$ is a smooth real manifold.
\end{coro*}\noindent This appears to be a new result even when $\A$ is acyclic.

One might consider taking a different class of cluster algebras and considering those cluster algebras which are \emph{locally} in that class.
\begin{itemize}
\item A cluster algebra $\A$ is \emph{tree-type} if it can come from an ice quiver $\Q$ whose mutable subquiver $\Q_{mut}$ is a union of trees.
\item A cluster algebra $\A$ is \emph{finite-type} if it can come from an ice quiver $\Q$ whose mutable subquiver $\Q_{mut}$ is a union of Dynkin quivers.\footnote{Finite-type cluster algebras are exactly those with finitely many clusters \cite{FZ03}.}
\item A cluster algebra $\A$ is \emph{$A$-type} if it can come from an ice quiver $\Q$ whose mutable subquiver $\Q_{mut}$ is a union of quivers of the form below.
\begin{center}
\begin{tikzpicture}[inner sep=0.5mm,scale=1,auto,minimum size=2mm]
	\node (1) at (-2,0)  [mutable] {};
	\node (2) at (-1,0)  [mutable] {};
	\node (3) at (0,0)  [mutable] {};
	\node (4) at (1,0)  [mutable] {};
	\node (5) at (2,0)  [mutable] {};
	\draw (1) to [-angle 90] (2);
	\draw (2) to [-angle 90] (3);
	\draw (3) to [dashed,-angle 90] (4);
	\draw (4) to [-angle 90] (5);
\end{tikzpicture}
\end{center}
\item A cluster algebra $\A$ is \emph{isolated} if it comes from an ice quiver $\Q$ in which there are no arrows between unfrozen vertices.
\end{itemize}
In fact, these are all locally equivalent.
\begin{thm*}[\textbf{\ref{thm: localequiv}}] Let $\A$ be a cluster algebra.  Then the following are equivalent.
\begin{itemize}
\item (Locally acyclic) $\X_A$ has a finite cover by acyclic cluster charts.
\item (Locally tree-type) $\X_A$ has a finite cover by tree-type cluster charts.
\item (Locally finite-type) $\X_A$ has a finite cover by finite-type cluster charts.
\item (Locally $A$-type) $\X_A$ has a finite cover by $A$-type cluster charts.
\item (Locally isolated) $\X_A$ has a finite cover by isolated cluster charts.
\end{itemize}
\end{thm*}

\subsection{Examples} The results of the previous section are only as interesting as the class of locally acyclic cluster algebras.  Thankfully, there are many more locally acyclic cluster algebras than acyclic cluster algebras.  Explicit examples can be found in Figure \ref{fig: LA}.

The main families of examples known are \emph{cluster algebras of marked surfaces}, which were introduced in \cite{GSV05} and developed in \cite{FST08}.  A \emph{marked surface} $(\S,\M)$ will be an oriented surface with boundary $\S$ together with a finite collection of marked points $\M$.  Barring some small examples, every marked surface has an associated cluster algebra $\A(\S,\M)$.  These provide large classes of locally acyclic cluster algebras.
\begin{thm*}[\textbf{\ref{thm: atleast2}}]
If $(\S,\M)$ has at least two boundary marked points in each component of $\S$, then $\A(\S,\M)$ is locally acyclic.
\end{thm*}
\begin{thm*}[\textbf{\ref{thm: inadisc}}]
If $\S$ includes into a disc, then $\A(\S,\M)$ is locally acyclic.
\end{thm*}
Since a relatively small class of marked surfaces give acyclic cluster algebras, this produces many examples of locally acyclic cluster algebras which are not acyclic (see Remark \ref{rem: nonA}).

Explicit examples of quivers of non-locally-acyclic cluster algebras can be found in Figure \ref{fig: nonLA}.  Marked surfaces also give families of cluster algebras which are not locally acyclic.
\begin{thm*}[\textbf{\ref{thm: noboundary}}]
If $\partial\S$ is empty, then $\A(\S,\M)$ is not locally acyclic.
\end{thm*}
\begin{thm*}[\textbf{\ref{thm: oneboundary}}]
If $|\M|=1$, then $\A(\S,\M)$ is not locally acyclic.
\end{thm*}

Examples are mostly from marked surfaces because there is a convenient topological method for constructing acyclic covers (Lemmas \ref{lemma: type1} and \ref{lemma: type2}).  There is no reason to suspect that local acyclicity is special to cluster algebras of marked surfaces.

\subsection{Constructing acyclic covers}

We also present some general techniques for constructing acyclic covers of $\X_\A$, on the level of quiver manipulations.  The most fundamental technique is the idea of \emph{covering pairs}.
\begin{defn}
A \textbf{covering pair} $(a_1,a_2)$ in an ice quiver $\Q$ is a pair of mutable vertices $a_1,a_2$ such that there is an arrow from $a_1$ to $a_2$ but this arrow is not in any bi-infinite path of mutable vertices.

A \textbf{covering pair} $(a_1,a_2)$ in a cluster algebra $\A$ is a pair of cluster variables such that there is some seed $(\Q,\a)$ containing $a_1,a_2$ and $(a_1,a_2)$ a covering pair in $\Q$.
\end{defn}
\begin{lemma*}[\textbf{\ref{lemma: covering}}]
Covering pairs in $\A$ are relatively prime.
\end{lemma*}
\noindent If $(a_1,a_2)$ is a covering pair in $\A$, then $\X_{\A[a_1^{-1}]}$ and $\X_{\A[a_2^{-1}]}$ cover $\X_\A$.  If $\A[a_1^{-1}]$ and $\A[a_2^{-1}]$ are also cluster localizations, then this is a cover by cluster charts.

This approach can be formally iterating by a branching, non-deterministic algorithm we call the \textbf{Banff algorithm}.  It takes as input a seed of $\A$, and it either fails, or produces the seeds of cluster localizations which define an acyclic cover of $\X_\A$.

\begin{framed}
\noindent \underline{Input}: a seed $(\Q,\mathbf{a})$.
\begin{enumerate}
\item If $(\Q,\mathbf{a})$ is equivalent to an acyclic seed, stop.
\item Otherwise, mutate $(\Q,\mathbf{a})$ to a seed $(\Q',\a')$ with a covering pair $(i_1,i_2)$.  If this is impossible, the algorithm \textbf{fails}.
\item Consider two seeds, which are the seed $(\Q',\a')$ with either $i_1$ or $i_2$ frozen.  Plug each of these seeds into Step 1.
\end{enumerate}
\noindent \underline{Output}: a finite set of acyclic seeds; or the algorithm \textbf{failed}.
\end{framed}

\begin{thm*}[\textbf{\ref{thm: Banff}}]
If the Banff algorithm outputs a finite set of acyclic seeds, these seeds define a finite acyclic cover of $\X_\A$, and so $\A$ is locally acyclic.
\end{thm*}
Figure \ref{fig: Banff2} contains an example where the Banff algorithm is used to show a quiver corresponds to a locally acyclic cluster algebra.

\begin{figure}
\begin{tikzpicture}[inner sep=0.5mm,scale=.35,auto,minimum size=2mm]
	\begin{scope}
	\node (1) at (0,0)  [red,mutable] {};
	\node (2) at (-2,0)  [mutable] {};
	\node (3) at (2,0)  [mutable] {};
	\node (4) at (-1,1.5)  [mutable] {};
	\node (5) at (1,1.5)  [mutable] {};
	\node (6) at (0,-1.5)  [red,mutable] {};
	\draw (1) to [-angle 90] (2);
	\draw (2) to [out=80, in=220,-angle 90] (4);
	\draw (2) to [out=40, in=260,-angle 90] (4);
	\draw (4) to [-angle 90] (1);
	\draw (1) to [-angle 90] (3);
	\draw (3) to [out=140, in=280,-angle 90] (5);
	\draw (3) to [out=100, in=320,-angle 90] (5);
	\draw (5) to [-angle 90] (1);
	\draw (1) to [red,-angle 90] (6);
	\end{scope}
    \begin{scope}[yshift=-1.5in]
        \draw[thick] (0,.5in) to (0,0) to (1.5in,0) to (1.5in,-.5in);
        \draw[thick] (0,.5in) to (0,0) to (-1.5in,0) to (-1.5in,-.5in);
    \end{scope}
	\begin{scope}[xshift=-1.5in,yshift=-3in]
	\node (1) at (0,0)  [frozen] {};
	\node (2) at (-2,0)  [mutable] {};
	\node (3) at (2,0)  [mutable] {};
	\node (4) at (-1,1.5)  [mutable] {};
	\node (5) at (1,1.5)  [mutable] {};
	\node (6) at (0,-1.5)  [mutable] {};
	\draw (1) to [-angle 90] (2);
	\draw (2) to [out=80, in=220,-angle 90] (4);
	\draw (2) to [out=40, in=260,-angle 90] (4);
	\draw (4) to [-angle 90] (1);
	\draw (1) to [-angle 90] (3);
	\draw (3) to [out=140, in=280,-angle 90] (5);
	\draw (3) to [out=100, in=320,-angle 90] (5);
	\draw (5) to [-angle 90] (1);
	\draw (1) to [-angle 90] (6);
	\end{scope}
        \node (a1) at (-1.5in,-4.25in) {Acyclic};
	\begin{scope}[xshift=1.5in,yshift=-3in]
	\node (1) at (0,0)  [mutable] {};
	\node (2) at (-2,0)  [mutable] {};
	\node (3) at (2,0)  [mutable] {};
	\node (4) at (-1,1.5)  [mutable] {};
	\node (5) at (1,1.5)  [mutable] {};
	\node (6) at (0,-1.5)  [frozen] {};
	\draw (1) to [-angle 90] (2);
	\draw (2) to [out=80, in=220,-angle 90] (4);
	\draw (2) to [out=40, in=260,-angle 90] (4);
	\draw (4) to [-angle 90] (1);
	\draw (1) to [-angle 90] (3);
	\draw (3) to [out=140, in=280,-angle 90] (5);
	\draw (3) to [out=100, in=320,-angle 90] (5);
	\draw (5) to [-angle 90] (1);
	\draw (1) to [-angle 90] (6);
	\end{scope}

    \node (a) at (3in,-3in) {$\Leftrightarrow$};

	\begin{scope}[xshift=4.5in,yshift=-3in]
	\node (1) at (0,0)  [red,mutable] {};
	\node (2) at (-2,0)  [mutable] {};
	\node (3) at (2,0)  [mutable] {};
	\node (4) at (-1,1.5)  [mutable] {};
	\node (5) at (1,1.5)  [red,mutable] {};
	\node (6) at (0,-1.5)  [frozen] {};
	\draw (1) to [angle 90-] (2);
	\draw (2) to [out=80, in=220,angle 90-] (4);
	\draw (2) to [out=40, in=260,angle 90-] (4);
	\draw (4) to [angle 90-] (1);
	\draw (1) to [-angle 90] (3);
	\draw (3) to [out=0, in=180,-angle 90,relative] (6);
	\draw (3) to [out=60, in=120,-angle 90,relative] (6);
	\draw (5) to [red,-angle 90] (1);
	\draw (1) to [angle 90-] (6);
	\end{scope}
    \begin{scope}[xshift=4.5in,yshift=-4.5in]
        \draw[thick] (0,.5in) to (0,0) to (1.5in,0) to (1.5in,-.5in);
        \draw[thick] (0,.5in) to (0,0) to (-1.5in,0) to (-1.5in,-.5in);
    \end{scope}
	\begin{scope}[xshift=3in,yshift=-6in]
	\node (1) at (0,0)  [red,mutable] {};
	\node (2) at (-2,0)  [mutable] {};
	\node (3) at (2,0)  [red,mutable] {};
	\node (4) at (-1,1.5)  [mutable] {};
	\node (5) at (1,1.5)  [frozen] {};
	\node (6) at (0,-1.5)  [frozen] {};
	\draw (1) to [angle 90-] (2);
	\draw (2) to [out=80, in=220,angle 90-] (4);
	\draw (2) to [out=40, in=260,angle 90-] (4);
	\draw (4) to [angle 90-] (1);
	\draw (1) to [red, -angle 90] (3);
	\draw (3) to [out=0, in=180,-angle 90,relative] (6);
	\draw (3) to [out=60, in=120,-angle 90,relative] (6);
	\draw (5) to [-angle 90] (1);
	\draw (1) to [angle 90-] (6);
	\end{scope}
	\begin{scope}[xshift=6in,yshift=-6in]
	\node (1) at (0,0)  [frozen] {};
	\node (2) at (-2,0)  [mutable] {};
	\node (3) at (2,0)  [mutable] {};
	\node (4) at (-1,1.5)  [mutable] {};
	\node (5) at (1,1.5)  [mutable] {};
	\node (6) at (0,-1.5)  [frozen] {};
	\draw (1) to [angle 90-] (2);
	\draw (2) to [out=80, in=220,angle 90-] (4);
	\draw (2) to [out=40, in=260,angle 90-] (4);
	\draw (4) to [angle 90-] (1);
	\draw (1) to [-angle 90] (3);
	\draw (3) to [out=0, in=180,-angle 90,relative] (6);
	\draw (3) to [out=60, in=120,-angle 90,relative] (6);
	\draw (5) to [-angle 90] (1);
	\draw (1) to [angle 90-] (6);
	\end{scope}
        \node (a2) at (6in,-7.25in) {Acyclic};

    \begin{scope}[xshift=3in,yshift=-7.5in]
        \draw[thick] (0,.5in) to (0,0) to (1.5in,0) to (1.5in,-.5in);
        \draw[thick] (0,.5in) to (0,0) to (-1.5in,0) to (-1.5in,-.5in);
    \end{scope}

	\begin{scope}[xshift=-1.5in,yshift=-9in]
	\node (1) at (0,0)  [mutable] {};
	\node (2) at (-2,0)  [mutable] {};
	\node (3) at (2,0)  [frozen] {};
	\node (4) at (-1,1.5)  [mutable] {};
	\node (5) at (1,1.5)  [frozen] {};
	\node (6) at (0,-1.5)  [frozen] {};
	\draw (1) to [-angle 90] (2);
	\draw (2) to [angle 90-] (4);
	\draw (4) to [-angle 90] (1);
	\draw (1) to [angle 90-] (3);
	\draw (3) to [out=60, in=120,-angle 90,relative] (6);
	\draw (5) to [angle 90-] (1);
	\draw (1) to [-angle 90] (6);
    \draw (4) to [-angle 90] (5);
    \draw (3) to [-angle 90] (5);
    \draw (6) to [out=30,in=150,relative,-angle 90] (4);
    \draw (2) to [out=-45,in=225,relative,-angle 90] (3);
	\end{scope}
    \node (b) at (0in,-9in) {$\Leftrightarrow$};
	\begin{scope}[xshift=1.5in,yshift=-9in]
	\node (1) at (0,0)  [mutable] {};
	\node (2) at (-2,0)  [mutable] {};
	\node (3) at (2,0)  [frozen] {};
	\node (4) at (-1,1.5)  [mutable] {};
	\node (5) at (1,1.5)  [frozen] {};
	\node (6) at (0,-1.5)  [frozen] {};
	\draw (1) to [angle 90-] (2);
	\draw (2) to [out=80, in=220,angle 90-] (4);
	\draw (2) to [out=40, in=260,angle 90-] (4);
	\draw (4) to [angle 90-] (1);
	\draw (1) to [-angle 90] (3);
	\draw (3) to [out=0, in=180,-angle 90,relative] (6);
	\draw (3) to [out=60, in=120,-angle 90,relative] (6);
	\draw (5) to [-angle 90] (1);
	\draw (1) to [angle 90-] (6);
	\end{scope}
        \node (a3) at (-1.5in,-10.25in) {Acyclic};
	\begin{scope}[xshift=4.5in,yshift=-9in]
	\node (1) at (0,0)  [frozen] {};
	\node (2) at (-2,0)  [mutable] {};
	\node (3) at (2,0)  [mutable] {};
	\node (4) at (-1,1.5)  [mutable] {};
	\node (5) at (1,1.5)  [frozen] {};
	\node (6) at (0,-1.5)  [frozen] {};
	\draw (1) to [angle 90-] (2);
	\draw (2) to [out=80, in=220,angle 90-] (4);
	\draw (2) to [out=40, in=260,angle 90-] (4);
	\draw (4) to [angle 90-] (1);
	\draw (1) to [-angle 90] (3);
	\draw (3) to [out=0, in=180,-angle 90,relative] (6);
	\draw (3) to [out=60, in=120,-angle 90,relative] (6);
	\draw (5) to [-angle 90] (1);
	\draw (1) to [angle 90-] (6);
	\end{scope}
        \node (a4) at (4.5in,-10.25in) {Acyclic};
\end{tikzpicture}
\caption{An example of the Banff algorithm.  The cluster variables have been surpressed, for space and because they do not affect acyclicity. Circles denote mutable vertices, squares denote frozen vertices, red indicates a covering pair defining the branching, and $\Leftrightarrow$ denotes mutation equivalence.
  }\label{fig: Banff2}
\end{figure}

\subsection{The structure of the paper}The structure of the paper follows.
\begin{enumerate}
\setcounter{enumi}{1}
\item \textbf{Preliminaries.} The necessary definitions of ice quivers, seeds and cluster algebras are given, as well as some standard results.  It is mostly review and notation-fixing.
\item \textbf{Locally acyclic cluster algebras.} The ideas of cluster localization and and locally acyclic cluster algebras are introduced.
\item \textbf{Properties of locally acyclic cluster algebras.} The first properties of locally acyclic cluster algebras are covered here.
\item \textbf{Covering pairs and the Banff algorithm.}  Covering pairs are introduced as a method for producing covers.  The Banff algorithm for constructing acyclic covers is introduced, as well as a reduced version for checking whether acyclic covers exist.
\item \textbf{Finite isolated covers.} As a consequence of the Banff algorithm, any locally acyclic cluster algebra can be covered by isolated cluster algebras.
\item \textbf{Regularity and rank.} Using the existence of finite isolated covers, the surjectivity of the exchange matrix implies the a locally acyclic cluster algebra is regular.
\item \textbf{Failure of local acyclicity.} This section explores how a cluster algebra can fail to be locally acyclic, in terms of combinatorial data, algebraic maps, and geometric behavior.
\item \textbf{Cluster algebras of marked surfaces.} This section reviews how to associate a cluster algebra to a suitable marked surface.  It mostly reviews \cite{FST08}.
\item \textbf{Marked surfaces and local acyclicity.} Covering pairs are reinterpreted topologically.  Several large classes of marked surface are shown to determine locally acyclic cluster algebras, and several other classes are shown to detemine non-locally-acyclic cluster algebras.
\item \textbf{Examples and non-examples.} Several concrete examples of locally acyclic and non-locally-acyclic cluster algebras are worked out in detail.  
\end{enumerate}

\subsection{Acknowledgements}

The author would like to thank a great many people for useful discussions and support; including A. Berenstein, M. Gekhtman, A. Knutson, G. Musiker, D. Thurston, M. Yakimov, and A. Zelevinsky. Special thanks also goes to the participants in LSU's Cluster Algebras seminar 2010-2011, particularly Neal Livesay, Jacob Matherne and Jesse Levitt; and to Frank Moore, for his patience with the author's constant questions about the computational algebra program Macaulay 2.

\section{Preliminaries}

Cluster algebras considered will be skew-symmetric of geometric type.  

\subsection{Quivers}

An \textbf{ice quiver} will be a quiver (without loops or directed 2-cycles), together with a distinguished subset of the vertices called \textbf{frozen vertices} (drawn as squares \begin{tikzpicture}\node (1) at (0,0) [frozen,inner sep=1mm] {};\end{tikzpicture}).  The unfrozen vertices are called \textbf{mutable} (drawn as circles \begin{tikzpicture}\node (1) at (0,0) [mutable] {};\end{tikzpicture}).  The number of vertices of an implicit ice quiver will often be denoted $n$.

For $\Q$ an ice quiver, $\Q_{mut}$ will denote the induced subquiver on the mutable vertices.  An ice quiver $\Q$ is called...
\begin{itemize}
\item ...\textbf{acyclic} if $\Q_{mut}$ has no directed cycles, and
\item ...\textbf{tree-type} if $\Q_{mut}$ is an orientation of a union of trees,
\item ...\textbf{finite-type} if $\Q_{mut}$ is an orientation of a union of Dynkin diagrams,
\item ...\textbf{$A$-type} if $\Q_{mut}$ is an orientation of a union of $A$-type Dynkin diagrams,
\item ...\textbf{isolated} if $\Q_{mut}$ has no arrows at all.
\end{itemize}


For an ice quiver $\Q$ with vertices $i$ and $j$, let
\[\Q_{ij} := \#(\text{arrows from $i$ to $j$}) - \#(\text{arrows from $i$ to $j$})\]
This defines an $n\times n$, skew-symmetric, integral matrix from which $\Q$ can be reconstructed.  


A \textbf{seed} $(\Q,\mathbf{a})$ of rank $n$ will be an ice quiver $\Q$ on $n$ vertices, together with a \textbf{cluster} $\mathbf{a}$: a map from the set of vertices of $\Q$ to $\mathbb{Q}(x_1,...,x_n)$.  Typically, the vertices will be numbered from $1$ to $n$, and so $\mathbf{a}$ can be defined as an ordered $n$-tuple $\{a_1,...,a_n\}$ in $\mathbb{Q}(x_1,...,x_n)$.  A seed is \textbf{initial} if $\mathbf{a}=\{x_1,...,x_n\}$.  A seed is \textbf{acyclic}, \textbf{isolated}, etc. if its ice quiver is.

\subsection{Cluster algebras}  

Given a seed $(\Q,\mathbf{a})$ with mutable vertices $\{1,...,m\}$, the \textbf{mutation at $k\in \{1,...,m\}$} of $(\Q,\a)$ is the new seed $(\mu_k(\Q),\mu_k(\a))$ defined by
\begin{itemize}
\item The new ice quiver $\mu_k(\Q)$ is defined by, for all $1\leq i, j\leq n$,
\[ \mu_k(\Q)_{ij} := \left\{\begin{array}{cc}
-\Q_{ij} & \text{if } k=i\text{ or }j\\
\Q_{ij} + \frac{|\Q_{ik}|\Q_{kj}+\Q_{ik}|\Q_{kj}|}{2} & \text{if } k\neq i,j
\end{array}\right\}\]
A vertex in $\mu_k(\Q)$ is frozen if the corresponding vertex in $\Q$ is frozen.
\item For $i\neq k$, set $\mu_k(a_i):=a_i$, and set
\[ \mu_k(a_k):=
\left(\prod_{j,\;\Q_{kj}>0} a_j^{\Q_{kj}}+\prod_{j,\;\Q_{kj}<0} a_j^{-\Q_{kj}}\right)a_k^{-1}\]
\end{itemize}
Mutation may be iterated at any sequence of vertices, and mutating at the same vertex twice returns to the original seed.


Given a seed $(\Q,\a)$, the \textbf{cluster algebra} $\A(\Q,\a)$ is the subalgebra of $\mathbb{Q}(x_1,...,x_n)$ generated by the \textbf{cluster variables}: the set of functions in $\mathbb{Q}(x_1,...,x_n)$ which occur in any seed mutation equivalent to $(\Q,\a)$.  We will write $\A(\Q)$ for $\A(\Q,\mathbf{x})$, where $\mathbf{x}=\{x_1,...,x_n\}$ is an initial cluster.  A cluster algebra is \textbf{acyclic}, \textbf{isolated}, etc. if has a seed of that type.

\subsection{The upper cluster algebra}

For a cluster algebra $\A$ with seed $(\Q,\mathbf{a})$ and $\mathbf{a}=\{a_1,...,a_n\}$, there is a Laurent embedding \cite[Theorem 3.1]{FZ02}
\[ \A\subset \Z[a_1^{\pm1},...,a_n^{\pm1}]\subset \mathbb{Q}(x_1,...,x_n)\]
The ring $\Z[a_1^{\pm1},...,a_n^{\pm1}]$ will be called the \textbf{Laurent ring} of the cluster $\mathbf{a}$.  The \textbf{upper cluster algebra} $\U$ is the intersection of the Laurent rings of each cluster in $\A$.  It follows that $\A\subseteq \U$.

Recall that a commutative domain is \textbf{integrally closed} if it is contains any root of a monic polynomial which exists in its fraction field.
\begin{prop}
The upper cluster algebra $\U$ is integrally closed.
\end{prop}
\begin{proof}
Let $f(x)$ be a monic polynomial in $\U$, and $a$ a root of $f(x)$ in the fraction field $\k(x_1,...,x_n)$.  For any cluster $\mathbf{a}$, its Laurent ring is integrally closed because it is a localization of a polynomial ring, so $a\in k[a_1^{\pm1},...,a_n^{\pm1}]$.  Therefore, $a\in \U$.
\end{proof}

\subsection{Krull dimension}

Cluster algebras have been defined over $\Z$, rather than over a field.  Some results are more natural over a field, and it will be convenient to occasionally tensor with $\mathbb{Q}$.  One such result is the (Krull) dimension.
\begin{prop}\label{prop: Krulldim}
For $\A$ rank $n$, $\mathbb{Q}\otimes \A$ and $\mathbb{Q}\otimes \U$ have dimension $n$.
\end{prop}
\begin{proof}
This is true for any algebra $R$ with
\[ \mathbb{Q}[x_1,...,x_n]\subseteq R \subset \mathbb{Q}[x_1^{\pm1},...,x_n^{\pm1}]\]
Any prime ideal $P$ in the Laurent ring is generated by its intersection with the polynomial ring, and therefore by its intersection with $R$.  It follows that no two primes in the Laurent ring have the same intersection with $R$, and so any chain of $n$ prime ideals in the Laurent ring gives a chain of $n$ prime ideals in $R$.  Then $dim(R)\geq n$.

Assume $R$ has a chain of $n+1$ prime ideals
\[ P_0\subsetneq P_1\subsetneq ...\subsetneq P_{n+1}\]
Choose a sequence of elements $r_i\in P_i-P_{i-1}$, and let $S\subset R$ be the subalgebra generated by $\{r_i\}$.  Then $Q_i:=S\cap P_i$ gives a chain of $n+1$ prime ideals in $S$; they are distinct because $Q_i$ contains $r_i$ but not $r_{i+1}$.  Then $S$ is a finitely-generated domain over $\mathbb{Q}$ with dimension $n+1$.

But the fraction field 
\[K(S)\subset K(R) = \mathbb{Q}(x_1,...,x_n)\]
and so it has transcendence degree $n$ over $\mathbb{Q}$.  By Theorem A \cite{Eis95}, this means that $S$ has Krull dimension $n$; this is a contradiction.  So $dim(R)=n$.
\end{proof}

\subsection{Acyclic cluster algebras}


Acyclic cluster algebras have many strong properties which are false in general.  For a seed $(\Q,\a)$, define
\[ \pi_i^+:=\prod_{j,\Q_{ij}>0}a_j^{\Q_{ij}},\;\;\;\pi_i^-:=\prod_{j,\Q_{ij}<0}a_j^{-\Q_{ij}}\]


\begin{lemma}\cite[Corollary 1.21]{BFZ05}\label{lemma: acyclicpres}
Let $(\Q,\mathbf{a})$ be an acyclic seed of $\A$, with frozen vertices numbered $m+1$ to $n$.  For $1\leq i\leq m$, let $a_i'$ denote the mutation of $a_i$ in $\A$.  Then
\[\A=\Z[a_1,a_2,...a_m,a_{m+1}^{\pm1},...,a_n^{\pm1}, a_1',a_2',...a_m']/ \langle a_1a_1'-\pi_1^+-\pi_1^-,...a_ma_m'-\pi_m^+-\pi_m^-\rangle\]
\end{lemma}
A commutative algebra 
is a \textbf{complete intersection} if it can be expressed as a quotient of a regular ring by a regular sequence.
\begin{coro}\label{coro: acyclic}
Let $\A$ be an acyclic cluster algebra.  Then $\A$ is finitely-generated and a complete intersection.
\end{coro}
\begin{proof}
Finite-generation is clear from the lemma.  
The ring \[\Z[a_1,a_2,...a_m,a_{m+1}^{\pm1},...,a_n^{\pm1}, a_1',a_2',...a_m']\] is a localization of a polynomial ring and so it is regular.  The relations given form a regular sequence in any order, because each one contains a unique variable of the form $a_i'$.
\end{proof}


Acyclic cluster algebras are known to coincide with their upper cluster algebra.

\begin{thm}\cite[Theorem 1.18]{BFZ05}\label{lemma: acyclicA=U}
Let $\A$ be an acyclic cluster algebra, and $\U$ is upper cluster algebra.  Then $\A=\U$.
\end{thm}

\subsection{The geometry of cluster algebras}



The goal of this paper is to use geometry to study algebraic properties of $\A$ and $\U$.  This will done by studying the affine schemes
\[\X_\A:=Spec(\A)\;\;\; \text{and}\;\;\;\X_\U:=Spec(\U).\]


Dual to the inclusion $\A\rightarrow \U$, there is a map of schemes
\[ \X_\U\rightarrow \X_\A\]
In many ways, this map behaves like a normalization map.  It is a dominant map which is an isomorphism on an open subscheme, and $\X_\U$ is a normal scheme.  


For each cluster $\mathbf{a}$ in $\A$, we have embeddings
\[ \A\hookrightarrow \U\hookrightarrow \Z[a_1^{\pm1},...,a_n^{\pm1}]\]
Geometrically, this is dual to an open inclusion of an algebraic torus.
\[ (\mathbb{A}^*_\Z)^n\hookrightarrow \X_\U\rightarrow \X_\A\]
The image of this inclusion in $\X_\A$ or $\X_\U$ is called a \textbf{cluster torus}, and it is a smooth, affine open subscheme.  The union of all cluster tori will be denoted $\X^\circ$; this is again a smooth open subscheme.\footnote{ The subscheme $\X^\circ$ is called the \textbf{cluster manifold}; see \cite{GSV03} and \cite{GSV10}.  However, this terminology can be misleading, since the smooth locus of $\X_A$ can be bigger.  Note there is no subscript on $\X^o$, since $\X^\circ_\A=\X^\circ_\U$ and the notation risks being cluttered.}

This gives a geometric definition of the upper cluster algebra.
\begin{prop}\label{prop: upperaffine}
The ring of global sections on the union of the cluster tori $\X^\circ$ is $\U$.
\end{prop}
\begin{proof}
As an open subscheme of $\X_\A$, every function on $\X^\circ$ lives in the field of fractions on $\X_\A$, which is $\mathbb{Q}(x_1,...,x_n)$.  For a given cluster $\mathbf{a}$, an element $f\in \Z(x_1,...,x_n)$ is a regular function on the corresponding cluster torus exactly when $f\in \Z[a_1^{\pm1},...,a_n^{\pm1}]$.  Therefore, $f$ is a regular function on all cluster tori exactly when it is in the intersection of the Laurent rings of all clusters.
\end{proof}
From this perspective, $\X_\U$ is the \emph{affinization} of $\X^\circ$; the universal affine scheme with a map from $\X^\circ$.

The union of the cluster tori represents the well-behaved part of the cluster scheme; it is a smooth quasi-affine variety.  If $\A$ is bad in some way, then this badness should be caused by the complement $\X^\circ-\X_\A$.  Therefore, careful analysis of this complement can be used to prove that the scheme $\X_\A$ and the algebra $\A$ have good properties.

The easiest way this can happen is when the complement is empty.
\begin{quest}
When does $\X_\A=\X^\circ$?
\end{quest}
Surprisingly, this appears to be a difficult question. As an example, for coefficient-free $A_n$, the union of the cluster tori and the cluster scheme coincide except when $n\equiv 3 \,(\text{mod }4)$, when the complement is a singular point.

\section{Locally acyclic cluster algebras}

A cluster algebra will have special localizations (called \emph{cluster localizations}) which are again cluster algebras.  On the level of quivers, these will correspond to some freezings of the original quiver (but not all freezings).  This enables the study of cluster algebras locally, by passing to a finite cover by simpler cluster algebras.

\subsection{Freezing and cluster localization}

For a seed $(\Q,\mathbf{a})$ of $\A$, let \[\{a_1,a_2,...,a_i\}\subset\mathbf{a}\] be a collection of mutable variables. Define the \textbf{freezing} of $(\Q',\mathbf{a})$ at $\{a_1,a_2,...,a_i\}$ to be the same seed as before, except that the set $\{a_1,...,a_i\}$ is now frozen.

For the rest of this section, let $\A'$ denote the cluster algebra associated to the seed $(A',\a)$; such a cluster algebra will be called a \textbf{freezing} of $\A$.  Let $\U'$ be the upper cluster algebra of $\A'$.  Like $\A$ and $\U$, the algebras $\A'$ and $\U'$ are a subalgebras of $\mathbb{Q}(x_1,...,x_n)$.

\begin{prop}\label{prop: inclusions}
There are natural inclusions in $\mathbb{Q}(x_1,...,x_n)$
\[\A'\subseteq \A[a_1^{-1},a_2^{-1},...,a_i^{-1}]\subseteq \U[a_1^{-1},a_2^{-1},...,a_i^{-1}]\subseteq \U'\]
\end{prop}
\begin{proof}
The cluster variables of ${\A}'$ are a subset of the cluster variables of $\A$.  The only new generators are the inverses of the newly-frozen variables, but those are in the localization by construction.  This gives the first inclusion.
Similarly, ${\A}'$ has fewer clusters than $\A$, so the intersection defining ${\U}'$ has strictly fewer terms than $\U$; so ${\U}\subseteq {\U}'$.  Since ${\U}'$ also contains the inverses of $\{a_1,...,a_i\}$, this gives the last inclusion.
The middle inclusion follows from the inclusion $\A\subseteq \U$.
\end{proof}

\begin{lemma}
Let $\A$ be a cluster algebra, and $\A'$ the freezing of $\A$ at the set $\{a_1,..,a_i\}$.  Then the following are equivalent.
\begin{enumerate}
\item $\A\subseteq \A'$, as subalgebras of $\mathbb{Q}(x_1,...,x_n)$.
\item $\A'=\A[a_1^{-1},...,a_i^{-1}]$, as subalgebras of $\mathbb{Q}(x_1,...,x_n)$.
\item Every cluster variable in $\A$ can be written as a polynomial in the cluster variables of $\A'$, divided by a monomial in the frozen variables of $\A'$.
\end{enumerate}
\end{lemma}
\begin{proof}
If $\A\subseteq\A'$, then $\A[a_1^{-1},...,a_n^{-1}]\subseteq\A'$. The reverse inclusion is by the Proposition, so $(1)$ implies $(2)$.  Statement $(2)$ immediately implies $(1)$, and $(3)$ is a restatement of $(1)$.
\end{proof}
\begin{defn}
A freezing $\A'$ of $\A$ is called a \textbf{cluster localization} if any of the equivalent conditions hold.
\end{defn}

%
%



Proposition \ref{prop: inclusions} gives a large class of freezings which are cluster localizations.
\begin{lemma}\label{lemma: acycliclocal}
If ${\A}'=\U'$, then $\A'$ is a cluster localization and \[\A'= \A[a_1^{-1},a_2^{-1},...,a_i^{-1}]= \U[a_1^{-1},a_2^{-1},...,a_i^{-1}]= \U'\]
\end{lemma}
\noindent A freezing $\A'$ which is \emph{not} a cluster localization is given in Proposition \ref{prop: notclusterlocal}.

%

Geometrically, a cluster localization gives an open inclusion $\X_{{\A}'}\subseteq \X_\A$ of schemes, which will be called a \textbf{cluster chart}.  The subscheme $\X_{\A'}$ is the subset of $\X_\A$ where the cluster variables $\{a_1,a_2,...,a_i\}$ are not zero.

\begin{rem}[\textbf{Warning!}] In much of what follows, cluster charts will be treated as a `good' collection of open sets in $\X_\A$.  However, the intersection of two cluster charts is not necessarily a cluster chart.  Therefore, they cannot be used as the basis of a topology on $\X_\A$, resulting in some counter-intuitive facts (see Remark \ref{rem: LAnonlocal}).
\end{rem}



\begin{rem}[Relation to Laurent embeddings]
 When $\{a_1,...,a_n\}$ is a full cluster, then $\A'$ is the Laurent ring of that cluster, and the inclusion $\A\subseteq \A'$ becomes the Laurent embedding.  In this way, cluster localizations can be regarded as generalizations of Laurent embeddings, and cluster charts generalize cluster tori.
%
\end{rem}

\subsection{Deleting vertices}
Freezing vertices is closely related to another quiver operation: \textbf{deleting} vertices.  Let $(\Q,\a)$ be a seed, with $\{a_1,..,a_i\}\subset\a$, and let $(\Q^\dag,\a^\dag)$ be the seed where the vertices corresponding to $\{a_1,...,a_i\}$ have been deleted (as well as any incident arrows).

\begin{prop}If all the vertices $\{a_1,...,a_i\}$ are frozen, there is an isomorphism
\[  \A(\Q,\a)/\langle a_1-1,...,a_i-1\rangle  \stackrel{\sim}{\longrightarrow} \A(\Q^\dag,\a^\dag)\]
\end{prop}
\begin{proof}
There is a natural bijection from the cluster variables of $\A(\Q,\a)$ to the cluster variables of $\A(\Q^\dag,\a^\dag)$.  The mutations will differ by the deleting extra terms in the products corresponding to frozen variables.  Since setting a variable to 1 is the same as deleting it from a product, and frozen variables only ever appear in products, the induced map is an isomorphism.
\end{proof}
However, if the variables are mutable, then the situation is similar to freezing.
\begin{prop}
There are inclusions
\[\A(\Q^\dag,\a^\dag)\subseteq \A(\Q,\a)/\langle a_1-1,...,a_i-1\rangle\subseteq \U(\Q,\a)/\langle a_1-1,...,a_i-1\rangle\subseteq \U(\Q^\dag,\a^\dag)\]
\end{prop}
\begin{proof}
Take Proposition \ref{prop: inclusions} for freezing $a_1,...,a_i$, and then set those variables to 1.
\end{proof}
These are all equality if $\A(\Q^\dag,\a^\dag)=\U(\Q^\dag,\a^\dag)$, but an example where
\[ \A(\Q^\dag,\a^\dag)\neq \A(\Q,\a)/\langle a_1-1,...,a_i-1\rangle\]
is given in Proposition \ref{prop: notclusterlocal}.

\subsection{Local acyclicity} We arrive at the main definition of this note.

\begin{defn}
A \textbf{locally acyclic cluster algebra} $\A$ is a cluster algebra such that $\X_\A$ has a finite cover
by acyclic cluster charts.
\end{defn}
Algebraically, this means there is a finite set of acyclic cluster localizations $\A\hookrightarrow \A_i$ such that $\A=\bigcap_i\A_i$ and every prime ideal in $\A$ has a non-trivial extension to some $\A_i$.  As we shall see, the key advantage of this approach is that any local property which is true for acyclic cluster algebras is true for locally acyclic cluster algebras.

This property is independent of any frozen variables in $\A$.
\begin{prop}
For $\Q$ an ice quiver, the cluster algebra $\A(\Q)$ is locally acyclic if and only if $\A(\Q_{mut})$ is.
\end{prop}
\begin{proof}
For an acyclic cover $\{\A_i\}$ of $\A(\Q)$, deleting the frozen variables in $\A_i$ which are frozen variables in $\A(\Q)$ gives an acyclic cover of $\A(\Q_{mut})$.  Similarly, any acyclic cover of $\A(\Q_{mut})$ gives an acyclic cover of $\A(\Q)$ by adding frozen variables.
\end{proof}

\begin{rem}\label{rem: LAnonlocal}
If $\A$ is locally acyclic, it does not follow that a freezing or cluster localization $\A'$ is also locally acyclic.  This is because cluster charts are not closed under intersection, and so an acyclic cover of $\X_\A$ cannot be intersected with $\X_{\A'}$ to give an acyclic cover of $\X_{\A'}$.

An example of a locally acyclic $\A$ with non-locally-acyclic freezing $\A'$ is given in Proposition \ref{prop: notclusterlocal}.  The author does not know of any examples locally acyclic $\A$ with non-locally-acyclic cluster localizations $\A'$.
\end{rem}

\subsection{Explicit examples and non-examples}

For the reader's benefit, here are some examples (proofs and further investigation can be found in Section \ref{section: examples}). 
\begin{figure}[h!]
\centering
\begin{tikzpicture}[scale=.5,auto,minimum size=2mm]
    \node (1) at (-1,-.58) [mutable] {};
    \node (2) at (1,-.58) [mutable] {};
    \node (3) at (0,1.15) [mutable] {};
    \node (4) at (0,0) [mutable] {};

    \draw (1) to [out=330,in=210,-angle 90] (2);
    \draw (2) to [out=90,in=330,-angle 90] (3);
    \draw (3) to [out=210,in=90,-angle 90] (1);
    \draw (1) to [-angle 90] (4);
    \draw (2) to [-angle 90] (4);
    \draw (3) to [-angle 90] (4);
\end{tikzpicture}
\hspace{.5in}
\begin{tikzpicture}[scale=.5,auto,minimum size=2mm]
    \node (1) at (-1,-.58) [mutable] {};
    \node (2) at (1,-.58) [mutable] {};
    \node (3) at (0,1.15) [mutable] {};
    \node (4) at (0,0) [mutable] {};
    \node (5) at (0,-2) [mutable] {};

    \draw (1) to [out=330,in=210,-angle 90] (2);
    \draw (2) to [out=90,in=330,-angle 90] (3);
    \draw (1) to [out=90,in=210,-angle 90] (3);
    \draw (4) to [-angle 90] (1);
    \draw (4) to [-angle 90] (2);
    \draw (3) to [out=300,in=60,-angle 90] (4);
    \draw (3) to [out=240,in=120,-angle 90] (4);
    \draw (1) to [out=270,in=150,angle 90-] (5);
    \draw (2) to [out=270,in=30,-angle 90] (5);
\end{tikzpicture}
\hspace{.5in}
\begin{tikzpicture}[scale=.5,auto,minimum size=2mm]
    \node (1) at (-2,1) [mutable] {};
    \node (2) at (-2,-1) [mutable] {};
    \node (3) at (-1,0) [mutable] {};
    \node (4) at (0,1) [mutable] {};
    \node (5) at (0,-1) [mutable] {};
    \node (6) at (1,0) [mutable] {};
    \node (7) at (2,1) [mutable] {};
    \node (8) at (2,-1) [mutable] {};

    \draw (1) to [-angle 90] (3);
    \draw (3) to [-angle 90] (2);
    \draw (3) to [-angle 90] (4);
    \draw (4) to [-angle 90] (6);
    \draw (6) to [-angle 90] (5);
    \draw (5) to [-angle 90] (3);
    \draw (8) to [-angle 90] (6);
    \draw (6) to [-angle 90] (7);
\end{tikzpicture}
\hspace{.5in}
\begin{tikzpicture}[inner sep=0.5mm,scale=.4,auto,minimum size=2mm]
	\node (1) at (0,0)  [mutable] {};
	\node (2) at (-2,0)  [mutable] {};
	\node (3) at (2,0)  [mutable] {};
	\node (4) at (-1,1.5)  [mutable] {};
	\node (5) at (1,1.5)  [mutable] {};
	\node (6) at (0,-1.5)  [mutable] {};
	\draw (1) to [-angle 90] (2);
	\draw (2) to [out=80, in=220,-angle 90] (4);
	\draw (2) to [out=40, in=260,-angle 90] (4);
	\draw (4) to [-angle 90] (1);
	\draw (1) to [-angle 90] (3);
	\draw (3) to [out=140, in=280,-angle 90] (5);
	\draw (3) to [out=100, in=320,-angle 90] (5);
	\draw (5) to [-angle 90] (1);
	\draw (1) to [-angle 90] (6);
\end{tikzpicture}
\caption{Some locally acyclic cluster algebras}
\label{fig: LA}
\end{figure}
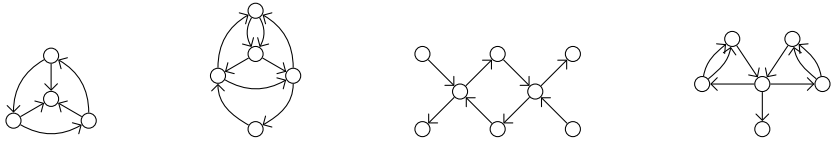

\begin{figure}[h!]
\begin{tikzpicture}[scale=.5,auto,minimum size=2mm]
	\node (1) at (0,1.15)  [mutable] {};
	\node (2) at (-1,-.58)  [mutable] {};
	\node (3) at (1,-.58)  [mutable] {};
	\draw (1) to [out=255, in=45,-angle 90] (2);
	\draw (2) to [out=15,in=165,-angle 90] (3);
	\draw (3) to [out=135,in=285,-angle 90] (1);
	\draw (1) to [out=210, in=90,-angle 90] (2);
	\draw (2) to [out=330,in=210,-angle 90] (3);
	\draw (3) to [out=90,in=330,-angle 90] (1);
\end{tikzpicture}
\hspace{.5in}
\begin{tikzpicture}[scale=.5,auto,minimum size=2mm]
    \node (1) at (-1,-.58) [mutable] {};
    \node (2) at (1,-.58) [mutable] {};
    \node (3) at (0,1.15) [mutable] {};
    \node (4) at (0,0) [mutable] {};

    \draw (1) to [out=330,in=210,-angle 90] (2);
    \draw (2) to [out=90,in=330,-angle 90] (3);
    \draw (1) to [out=90,in=210,-angle 90] (3);
    \draw (4) to [-angle 90] (1);
    \draw (4) to [-angle 90] (2);
    \draw (3) to [out=300,in=60,-angle 90] (4);
    \draw (3) to [out=240,in=120,-angle 90] (4);
\end{tikzpicture}
\hspace{.5in}
\begin{tikzpicture}[scale=.3,auto,minimum size=2mm]
    \node (1) at (0,1.14) [mutable] {};
    \node (2) at (-1,-.57) [mutable] {};
    \node (3) at (1,-.57) [mutable] {};
    \node (4) at (-2,1.14) [mutable] {};
    \node (5) at (2,1.14) [mutable] {};
    \node (6) at (0,-2.28) [mutable] {};

    \draw (1) to [-angle 90] (2);
    \draw (2) to [-angle 90] (3);
    \draw (3) to [-angle 90] (1);
    \draw (4) to [-angle 90] (1);
    \draw (1) to [-angle 90] (5);
    \draw (5) to [-angle 90] (3);
    \draw (3) to [-angle 90] (6);
    \draw (6) to [-angle 90] (2);
    \draw (2) to [-angle 90] (4);
    \draw (4) to [angle 90-,out=45,in=135] (5);
    \draw (5) to [angle 90-,out=285,in=15] (6);
    \draw (6) to [angle 90-,out=165,in=255] (4);
\end{tikzpicture}
\hspace{.5in}
\begin{tikzpicture}[inner sep=0.5mm,scale=.4,auto,minimum size=2mm]
	\node (1) at (0,0)  [mutable] {};
	\node (2) at (-2,0)  [mutable] {};
	\node (3) at (2,0)  [mutable] {};
	\node (4) at (-1,1.5)  [mutable] {};
	\node (5) at (1,1.5)  [mutable] {};
	\node (6) at (-1,-1.5)  [mutable] {};
	\node (7) at (1,-1.5)  [mutable] {};
	\draw (1) to [-angle 90] (2);
	\draw (2) to [out=80, in=220,-angle 90] (4);
	\draw (2) to [out=40, in=260,-angle 90] (4);
	\draw (4) to [-angle 90] (1);
	\draw (1) to [angle 90-] (3);
	\draw (3) to [out=140, in=280,angle 90-] (5);
	\draw (3) to [out=100, in=320,angle 90-] (5);
	\draw (5) to [angle 90-] (1);
	\draw (1) to [-angle 90] (7);
	\draw (7) to [out=160, in=20,-angle 90] (6);
	\draw (7) to [out=200, in=-20,-angle 90] (6);
	\draw (6) to [-angle 90] (1);
\end{tikzpicture}
\caption{Some non-locally-acyclic cluster algebras}
\label{fig: nonLA}
\end{figure}
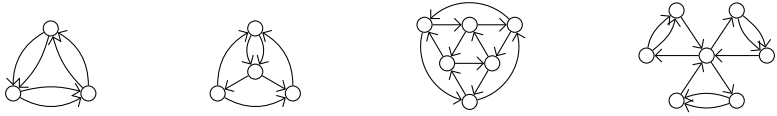

\section{Basic properties of locally acyclic cluster algebras}\label{section: properties}
 Locally acyclic cluster algebras will have many properties which follow from locality.

\subsection{The upper cluster algebra}  The first, and perhaps most interesting, property of locally acyclic cluster algebras is that they are equal to their own upper cluster algebra.
\begin{thm}\label{thm: A=U}
If $\A$ is locally acyclic, then $\A=\U$.
\end{thm}
\begin{proof}
Let $\{\X_{\A_i}\}$ be a collection of acyclic charts which cover $\X_\A$, so that the $\A_i$ are acyclic localizations of $\A$.  By Lemma \ref{lemma: acycliclocal},
\[ \A_i\otimes_\A\U=\A[a_1^{-1},...,a_i^{-1}]\otimes_\A\U=\U[a_1^{-1},..,a_i^{-1}]=\U_i\]
Geometrically, there is a pullback diagram
\[\begin{array}{ccc}
\X_{\U_i} & \hookrightarrow & \X_\U \\
\downarrow & & \downarrow \\
\X_{\A_i} & \hookrightarrow & \X_\A
\end{array}\]
Since $\A_i$ is acyclic, $\X_{\U_i}\stackrel{\sim}{\longrightarrow} \X_{\A_i}$ by Lemma \ref{lemma: acycliclocal}.  Therefore, the map $\X_\U\rightarrow \X_\A$ is locally an isomorphism in the target, so it is an isomorphism.
\end{proof}

\subsection{Algebraic properties} Local acyclicity also implies several algebraic properties which are local in nature.
\begin{thm}\label{thm: LAprops}
If $\A$ is locally acyclic, then $\A$ is finitely-generated, integrally closed, and locally a complete intersection.
\end{thm}
\begin{proof}
Acyclic cluster schemes are finite-type over $Spec(\mathbb{Z})$, by Lemma \ref{lemma: acyclicpres}.  Since $\X_\A$ can be covered by finitely many finite-type schemes, it follows from \cite[Exercise II.3.3(b)]{Har77} that $\A$ is finitely-generated.  Since $\A=\U$, $\A$ is integrally closed.  Being locally a complete intersection is a local property by definition.
\end{proof}

\subsection{The Weil-Petersson form}\label{section: WP}

 In \cite{GSV05}, the authors define a natural 2-form $\omega_{WP}$ on the union of the cluster tori $\X^\circ$, called the \textbf{Weil-Petersson} form, which encodes the information of the exchange matrix.  For $\mathbf{a}=\{a_1,...,a_n\}$ a cluster with exchange matrix $A$, the Weil-Petersson form is defined on the cluster torus of $\mathbf{a}$ by
\[ \omega_{WP}:= \sum_{i,j} A_{ij}\frac{dx_i\wedge dx_j}{x_ix_j}\]
It can then be shown that this definition is independent of the choice of cluster torus, and so it extends to all of $\X^\circ$.

One would hope that $\omega_{WP}$ would extend to $\X_\A$.  However, $\X_\A$ may be singular, even in the acyclic case, and so it is necessary to pass to a generalized notion of differential form, which is well-defined at singularities.  This is accomplished by passing to \emph{K\"ahler differential forms}, whose definition can be found in \cite[II.8]{Har77} or \cite{GregWP}.

\begin{lemma}\cite[Theorem 5.2.2.]{GregWP}
If $\A$ is acyclic, the Weil-Petersson form $\omega_{WP}$ extends to a K\"ahler 2-form on $\X_\A$.
\end{lemma}

Since K\"ahler 2-forms form a sheaf on $\X_\A$, they can always be patched together from local data.

\begin{thm}\label{thm: WP} If $\A$ is locally acyclic, the Weil-Petersson form $\omega_{WP}$ extends to a K\"ahler 2-form on $\X_\A$.
\end{thm}

An algebraic interpretation of this says that it is possible to find $a_i,b_i,c_i\in\A$ such that
\[ \widetilde{\omega_{WP}}:= \sum_i a_i db_i\wedge dc_i\]
agrees with $\omega_{WP}$ on $\X^\circ$.
\begin{rem}
The paper \cite{GregWP} does not produce such an expression.  It constructs an explicit expression locally around any point in $\X_\A$, but it fails to produce a global expression (though such a form exists).
\end{rem}
\begin{rem}
This fact is best appreciated in comparison with a similar form which fails to extend.  On any cluster torus, there are non-zero $n$-forms $\omega_{vol}$ which are invariant for the $n$-torus action, and such $n$-forms differ by a scaling.  Such a form extends to a well-defined $n$-form $\omega_{vol}$ on $\X^\circ$.  However, in the case of the $A_3$ with no frozen vertices, this $3$-form does not extend to $\X_\A$.
%
\end{rem}

%

\section{Covering pairs and the Banff algorithm}

This section develops some important techniques for producing covers of cluster schemes.

%


\subsection{Covering pairs} To cover $\X_\A$ with cluster charts, it will be necessary when pairs of cluster variables are relatively prime to each other.

Let $\Q$ be an ice quiver.  A \textbf{(directed) bi-infinite path} in $\Q$ is a sequence of mutable vertices $i_k$ indexed by $k\in\mathbb{Z}$ such that $\Q_{i_k,i_{k+1}}>0$ for all $k$; that is, there is at least one arrow from $i_k$ to $i_{k+1}$.  We will say an arrow $e$ is in a bi-infinite path if there is a bi-infinite path $\{i_k\}_{k\in \mathbb{Z}}$ with $i_1$ the source of $e$, and $i_2$ the target.  A cycle determines a bi-infinite path, but an arrow can be in a bi-infinite path without being in a cycle.
\begin{prop}
An arrow $e$ is in a bi-infinite path iff there is a cycle upstream from $e$ and a cycle downstream from $e$.\footnote{These cycles may coincide.}
\end{prop}
\begin{proof}
Since the set of mutable vertices is finite, there must be $1\leq k\leq l$ such that $i_k=i_l$; then $\{i_k,i_{k+1},...,i_l\}$ is a cycle upstream from from $e$.  The other direction is similar.
\end{proof}

\begin{defn}
A \textbf{covering pair} $(a_1,a_2)$ in an ice quiver $\Q$ is a pair of mutable vertices $a_1,a_2$ such that there is an arrow from $a_1$ to $a_2$ but this arrow is not in any bi-infinite path (in $\Q_{mut}$).

A \textbf{covering pair} $(a_1,a_2)$ in a cluster algebra $\A$ is a pair of cluster variables such that there is some seed $(\Q,\a)$ with $a_1,a_2\in \a$ and $(a_1,a_2)$ a covering pair in $\Q$.
\end{defn}

The following is the key lemma for finding pairs of cluster variables which do not simultaneously vanish on $\X_\A$.
\begin{lemma}\label{lemma: covering}
Covering pairs are relatively prime in $\A$.
\end{lemma}

\begin{proof}Let $(a_{i_1},a_{i_2})$ be a covering pair.
Assume for contradiction that $P$ is a prime ideal with $a_{i_1},a_{i_2}\in I$.  The mutation relation at $i_2$ is
\[ a_{i_2}a_{i_2}'=\prod_{j,\;\Q_{i_2j}>0} a_j^{\Q_{i_2j}}+\prod_{j,\;\Q_{i_2j}<0} a_j^{-\Q_{i_2j}}\]
Since $\Q_{i_1i_2}>0$, the second term on the right hand side is in $P$.  Since the left hand side is in $P$, it follows that $\prod_{j,\;\Q_{i_2j}>0} a_j^{\Q_{i_2j}}\in P$.  By primality, there is some mutable $j$ with $\Q_{i_2j}>0$ such that $a_j\in P$.

Set $i_3:=j$, and repeat this argument for the pair $(i_2,i_3)$.  Iterating this produces a list of mutable vertices $i_1,i_2,...$ such that $\Q_{i_k,i_{k+1}}>0$.  Similarly, this argument may be iterated in the opposite direction to produce a bi-infinite path $...i_{-1},i_0,i_1,i_2,...$ of mutable indices.  This contradicts the second condition.  Therefore, if $i_1$ and $i_2$ are as above, then $a_{i_1}$ and $a_{i_2}$ cannot be in the same prime ideal.
\end{proof}
Geometrically, for a covering pair $(a_{i_1},a_{i_2})$, at most one of $a_{i_1}$ and $a_{i_2}$ can vanish at any point in $\X_\A$.  This means that $\X_\A$ is the union of the corresponding localizations.
\begin{coro}
Let $(a_{i_1},a_{i_2})$ be a covering pair.  Then $\X_{\A[a_{i_1}^{-1}]}$ and $\X_{\A[a_{i_2}^{-1}]}$ are two open subschemes which cover $\X_\A$.
\end{coro}
If $\A[a_{i_1}^{-1}]$ and $\A[a_{i_2}^{-1}]$ are cluster localizations, then these localizations are the same as freezing $a_{i_1}$ or $a_{i_2}$ (respectively), and thus this would give two cluster charts which cover $\X_\A$.  This will be our main technique for producing coverings by cluster charts.

\subsection{The Banff algorithm} The covering lemma gives a direct method of attempting to produce a locally acyclic cover of a cluster algebra $\A$.  At each step, one looks for a covering pair in any seed of the cluster algebra.  If this is possible, then one can refine the cover into simpler pieces, by consider the two quivers corresponding to freezing each part of a covering pair.  Iterating this procedure will either produce a locally acyclic cover of $\A$, or it will hit a `dead end' in the form of a seed with no covering pairs.

We encode this in a branching algorithm, affectionately called the \textbf{Banff algorithm}, which attempts to produce a finite acyclic cover of $\X_\A$.

\begin{framed}
\noindent \underline{Input}: a seed $(\Q,\mathbf{a})$.
\begin{enumerate}
\item If $(\Q,\mathbf{a})$ is equivalent to an acyclic seed, stop.
\item Otherwise, mutate $(\Q,\mathbf{a})$ to a seed $(\Q',\a')$ with a covering pair $(i_1,i_2)$.  If this is impossible, the algorithm \textbf{fails}.
\item Consider two seeds, which are the seed $(\Q',\a')$ with either $i_1$ or $i_2$ frozen.  Plug each of these seeds into Step 1.
\end{enumerate}
\noindent \underline{Output}: a finite set of acyclic seeds; or the algorithm \textbf{failed}.
\end{framed}

If the algorithm does not fail, it will terminate in finitely-many steps, since the number of mutable indices drops at each iteration.  However, this does not mean it will take finite time, since the absence of a covering pair cannot always be checked in finite time.  Also, the algorithm is far from deterministic; that is, it can produce different outputs from the same input.
\begin{thm}\label{thm: Banff}
If the Banff algorithm outputs a finite set of acyclic seeds, these seeds define a finite acyclic cover of $\X_\A$, and so $\A$ is locally acyclic.
\end{thm}
\begin{proof}
The output of the algorithm is a finite set of seeds, $\{(\Q_i,\mathbf{a}_i)\}$, each of which can be obtained from the original seed $(\Q,\mathbf{a})$ first by mutation, and then by freezing a collection of vertices.  Since $(\Q_i,\mathbf{a}_i)$ is an acyclic seed, by Lemma \ref{lemma: acycliclocal}, the corresponding cluster algebra $\A_i$ is the localization of $\A$ at the set of newly-frozen variables.

Now let $P$ be a prime ideal in $\A$.  At each branch of the algorithm, at least one of $a_{i_1}$ and $a_{i_2}$ is not in $P$.  Inductively, there is some seed $(A_i,\mathbf{a}_i)$ in the output such that $P$ does not contain any of the newly-frozen variables.  Then $\X_{\A_i}$ is an open neighborhood of $P$ in $\X_\A$.  Since $P$ was arbitrary, the $\X_{\A_i}$ cover $\X_\A$, and so $\A$ is locally acyclic.
\end{proof}

\begin{ex}
Let $\Q$ be the quiver at the top of Figure \ref{fig: Banff1}.
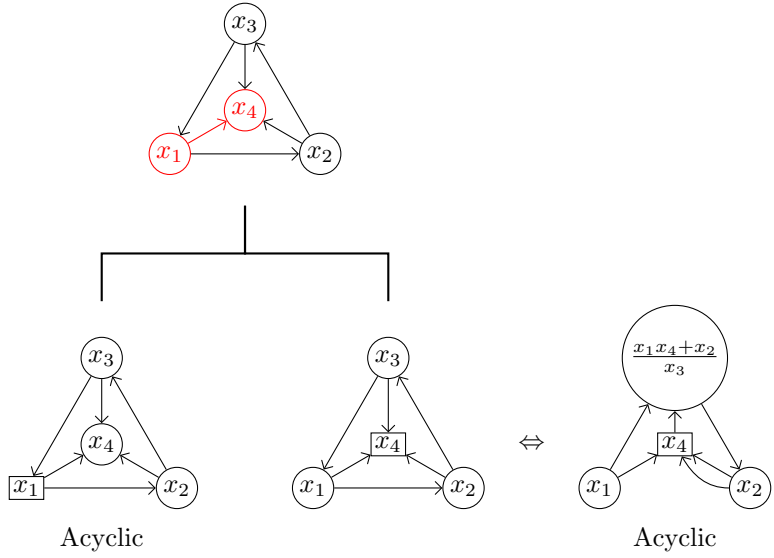
\begin{figure}
\begin{tikzpicture}[scale=1,auto]
\begin{scope}
    \node (1) at (-1,-.58) [red, mutable] {$x_1$};
    \node (2) at (1,-.58) [mutable] {$x_2$};
    \node (3) at (0,1.15) [mutable] {$x_3$};
    \node (4) at (0,0) [red, mutable] {$x_4$};

    \draw (1) to [-angle 90] (2);
    \draw (2) to [-angle 90] (3);
    \draw (3) to [-angle 90] (1);
    \draw (1) to [red,-angle 90] (4);
    \draw (2) to [-angle 90] (4);
    \draw (3) to [-angle 90] (4);
\end{scope}
\begin{scope}[yshift=-.75in,scale=.5]
        \draw[thick] (0,.5in) to (0,0) to (1.5in,0) to (1.5in,-.5in);
        \draw[thick] (0,.5in) to (0,0) to (-1.5in,0) to (-1.5in,-.5in);
\end{scope}
\begin{scope}[yshift=-1.75in,xshift=-.75in,scale=1,auto]
    \node (1) at (-1,-.58) [frozen] {$x_1$};
    \node (2) at (1,-.58) [mutable] {$x_2$};
    \node (3) at (0,1.15) [mutable] {$x_3$};
    \node (4) at (0,0) [mutable] {$x_4$};

    \draw (1) to [-angle 90] (2);
    \draw (2) to [-angle 90] (3);
    \draw (3) to [-angle 90] (1);
    \draw (1) to [-angle 90] (4);
    \draw (2) to [-angle 90] (4);
    \draw (3) to [-angle 90] (4);
\end{scope}
    \node (a1) at (-.75in,-2.25in) {Acyclic};
\begin{scope}[yshift=-1.75in,xshift=.75in,scale=1,auto]
    \node (1) at (-1,-.58) [mutable] {$x_1$};
    \node (2) at (1,-.58) [mutable] {$x_2$};
    \node (3) at (0,1.15) [mutable] {$x_3$};
    \node (4) at (0,0) [frozen] {$x_4$};

    \draw (1) to [-angle 90] (2);
    \draw (2) to [-angle 90] (3);
    \draw (3) to [-angle 90] (1);
    \draw (1) to [-angle 90] (4);
    \draw (2) to [-angle 90] (4);
    \draw (3) to [-angle 90] (4);
\end{scope}
		\node (=) at (1.5in,-1.75in) {$\Leftrightarrow$};
\begin{scope}[yshift=-1.75in,xshift=2.25in,scale=1,auto]
    \node (1) at (-1,-.58) [mutable] {$x_1$};
    \node (2) at (1,-.58) [mutable] {$x_2$};
    \node (3) at (0,1.15) [mutable] {$\frac{x_1x_4+x_2}{x_3}$};
    \node (4) at (0,0) [frozen] {$x_4$};

    \draw (2) to [angle 90-] (3);
    \draw (3) to [angle 90-] (1);
    \draw (1) to [-angle 90] (4);
    \draw (2) to [-angle 90] (4);
    \draw (2) to [-angle 90,out=180,in=300] (4);
    \draw (3) to [angle 90-] (4);
\end{scope}
    \node (a2) at (2.25in,-2.25in) {Acyclic};
\end{tikzpicture}
\caption{A successful implimentation of the Banff algorithm}
\label{fig: Banff1}
\end{figure}
The arrow from $x_1$ to $x_4$ is not in any bi-infinite path, so $(x_1,x_4)$ are a covering pair, and relatively prime in $\A(\Q)$ by Lemma \ref{lemma: covering}.  Each branch of the algorithm freezes one of these variables.  Since they are both equivalent to acyclic seeds, the algorithm produces a cover of $\X_\A$ by two acyclic charts.
%
\end{ex}
A more elaborate example was given as Figure \ref{fig: Banff2} in the Introduction.

\begin{rem}
The converse of the theorem (the failure of the Banff algorithm implies not locally acyclic) is not true. For an example of a locally acyclic cluster algebra for which the Banff algorithm can fail, but does not always fail, see Remark \ref{rem: Banfffail}.
\end{rem}


\begin{rem}
One unfortunate complication is that, in general, freezing a cluster variable is not the same as inverting it.  At the intermediate steps of the Banff algorithms (when some vertices have been frozen but the seed is not yet acyclic), it is not known that these describe cluster charts in $\X_\A$ (ie, localizations of cluster variables).  It is only when the seeds become acyclic that the two notions (freezing and localization) are known to coincide.  If the Banff algorithm does not fail, then every intermediate seed which appeared corresponds to a locally acyclic cluster algebra, and so by Lemma \ref{lemma: acycliclocal}, they correspond to cluster charts in $\X_\A$.   However, if the Banff algorithm fails, the seed $(\Q',\a')$ at which it fails may not correspond to an open subscheme in $\X_\A$.  This is an annoying conceptual stumbling block, but it ultimately does not affect the effectiveness of the algorithm at producing acyclic covers.
\end{rem}

\subsection{The reduced Banff algorithm}

When simply checking \emph{if} a cluster algebra is locally acyclic (instead of explicitly producing the acyclic cover), the Banff algorithm may be simplified.  This is because the acyclicity of a seed $(\Q,\a)$ only depends on the mutable subquiver $\Q_{mut}$.  Therefore,
\begin{itemize}
\item The cluster may be forgotten, leaving only the quivers.
\item Frozen variables may be deleted.
\end{itemize}
Furthermore, if $\X_\A$ can be covered by locally acyclic cluster charts, it can be covered by acyclic cluster charts.  Therefore, it will suffice to stop the branches whenever a quiver is known to determine a locally acyclic cluster algebra.

The simplified version of the algorith, the \textbf{reduced Banff algorithm}, is given below.

\begin{framed}
\noindent \underline{Input}: a quiver $\Q$ with no frozen vertices.
\begin{enumerate}
\item If $\Q$ determines a locally acyclic cluster algebra, stop.
\item Otherwise, mutate $\Q$ to a quiver $\Q'$ with a covering pair $(i_1,i_2)$.  If this is impossible, the algorithm \textbf{fails}.
\item Consider two quivers, which are the quiver $\Q'$ with either $i_1$ or $i_2$ deleted.  Plug each of these seeds into Step 1.
\end{enumerate}
\noindent \underline{Output}: a finite set of locally acyclic quivers; or the algorithm \textbf{failed}.
\end{framed}

\begin{prop}\label{prop: rBanff}
If the reduced Banff algorithm produces a finite set of locally acyclic quivers, then $\A(\Q)$ is locally acyclic.
\end{prop}
The downside of this simplification is that it does not compute what the open cover of $\X_{\A(\Q)}$ is; merely that it exists.

\begin{figure}[h]
\begin{tikzpicture}[scale=.5,auto]
\begin{scope}
    \node (1) at (-1,-.58) [red, mutable] {};
    \node (2) at (1,-.58) [mutable] {};
    \node (3) at (0,1.15) [mutable] {};
    \node (4) at (0,0) [red, mutable] {};

    \draw (1) to [-angle 90] (2);
    \draw (2) to [-angle 90] (3);
    \draw (3) to [-angle 90] (1);
    \draw (1) to [red,-angle 90] (4);
    \draw (2) to [-angle 90] (4);
    \draw (3) to [-angle 90] (4);
\end{scope}
\begin{scope}[yshift=-.75in,scale=.5]
        \draw[thick] (0,.5in) to (0,0) to (1.5in,0) to (1.5in,-.5in);
        \draw[thick] (0,.5in) to (0,0) to (-1.5in,0) to (-1.5in,-.5in);
\end{scope}
\begin{scope}[yshift=-1.75in,xshift=-.75in,scale=1,auto]
    \node (2) at (1,-.58) [mutable] {};
    \node (3) at (0,1.15) [mutable] {};
    \node (4) at (0,0) [mutable] {};

    \draw (2) to [-angle 90] (3);
    \draw (2) to [-angle 90] (4);
    \draw (3) to [-angle 90] (4);
\end{scope}
    \node (a1) at (-.75in,-2.25in) {Acyclic};
\begin{scope}[yshift=-1.75in,xshift=.75in,scale=1,auto]
    \node (1) at (-1,-.58) [mutable] {};
    \node (2) at (1,-.58) [mutable] {};
    \node (3) at (0,1.15) [mutable] {};

    \draw (1) to [-angle 90] (2);
    \draw (2) to [-angle 90] (3);
    \draw (3) to [-angle 90] (1);
\end{scope}
		\node (=) at (1.5in,-1.75in) {$\Leftrightarrow$};
\begin{scope}[yshift=-1.75in,xshift=2.25in,scale=1,auto]
    \node (1) at (-1,-.58) [mutable] {};
    \node (2) at (1,-.58) [mutable] {};
    \node (3) at (0,1.15) [mutable] {};

    \draw (2) to [angle 90-] (3);
    \draw (3) to [angle 90-] (1);
\end{scope}
    \node (a2) at (2.25in,-2.25in) {Acyclic};
\end{tikzpicture}
\caption{The reduced version of Figure \ref{fig: Banff1}}
\label{fig: RBanff}
\end{figure}
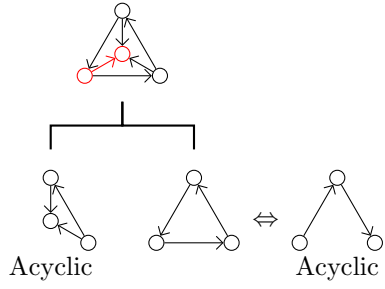

\section{Finite isolated covers}

As we will see, a finite acyclic cover of $\X_\A$ can be refined to a finite isolated cover of $\X_\A$.  Since isolated cluster algebras are easy to work with, this will have additional implications for locally acyclic cluster algebras.

\subsection{Refining acyclic covers}

It is not necessary to stop the Banff algorithm at an acyclic seed.  If $\Q$ is an acyclic quiver, it has no bi-infinite paths, so either it has a covering pair or it is an isolated quiver.  If we apply the Banff algorithm to $\Q$, stopping only when a seed is isolated, we can always produce a cover by isolated cluster charts.

\begin{lemma}\label{lemma: isolatedcover}
Let $(\Q,\a)$ be an acyclic seed for a cluster algebra $\A$.  Then there is a finite collection of freezings $\{(\Q_i,\a_i)\}$ of $(\Q,\a)$ which are isolated, and such that $\X_{\A(\Q_i,\a_i)}$ covers $\X_\A$.
\end{lemma}
\begin{proof}
If $\Q$ is not isolated, there is some edge $e$ in $\Q$ between two mutable vertices $i_1$ and $i_2$ which is not in a  bi-infinite path.  Then by Lemma \ref{lemma: covering}, $a_{i_1}$ and $a_{i_2}$ are relatively prime in $\A$.  Therefore, $\X_{\A}$ is covered by the acyclic cluster localizations $\X_{\A[a_{i_1}^{-1}]}$ and $\X_{\A[a_{i_2}^{-1}]}$.

Therefore, any acyclic, non-isolated cluster scheme can be covered by two acyclic cluster scheme with fewer mutable vertices.  Iterating this argument gives the desired cover.
\end{proof}
\begin{rem}
The statement is slightly stronger than the existence of an isolated cover.  It asserts that an isolated cover may be obtained \emph{without} mutating the quiver, so that all freezings are of the same seed.
\end{rem}

\begin{coro}\label{coro: isolatedcover}
If $\A$ is an acyclic cluster algebra, then $\X_\A$ has a finite cover by isolated cluster charts.  
\end{coro}
%
\begin{rem}\label{rem: subcover}
In some cases, it is even possible to cover an isolated cluster algebra by proper cluster localizations.  For the quiver below (which is a cluster localization of $A_5$), the three mutable vertices correspond to cluster variables which generate the trivial ideal.
\begin{center}
\begin{tikzpicture}[inner sep=0.5mm,scale=.5,auto,minimum size=2mm]
	\node (1) at (-2,0)  [mutable] {};
	\node (2) at (-1,0)  [frozen] {};
	\node (3) at (0,0)  [mutable] {};
	\node (4) at (1,0)  [frozen] {};
	\node (5) at (2,0)  [mutable] {};
	\draw (1) to [-angle 90] (2);
	\draw (2) to [-angle 90] (3);
	\draw (3) to [-angle 90] (4);
	\draw (4) to [-angle 90] (5);
\end{tikzpicture}
\end{center}
The corresponding scheme $\X_\A$ can be covered by three proper cluster charts, each corresponding to freezing a mutable variable.
\end{rem}


\subsection{Relation to other local types}

One might ask whether it is interesting to take another class of cluster algebras (besides acyclics), and consider those cluster algebras which are \emph{locally} in this class.  However, for many classes of cluster algebras, these local definitions coincide.
\begin{thm}\label{thm: localequiv} Let $\A$ be a cluster algebra.  Then the following are equivalent.
\begin{itemize}
\item (Locally acyclic) $\X_A$ has a finite cover by acyclic cluster charts.
\item (Locally tree-type) $\X_\A$ has a finite cover by tree-type cluster charts.
\item (Locally finite-type) $\X_A$ has a finite cover by finite-type cluster charts.
\item (Locally $A$-type) $\X_A$ has a finite cover by $A$-type cluster charts.
\item (Locally isolated) $\X_A$ has a finite cover by isolated cluster charts.
\end{itemize}
\end{thm}
\begin{proof}
Since each class of algebra is increasingly more restrictive, it suffices to know that locally acyclic implies locally isolated (Corollary \ref{coro: isolatedcover}).
\end{proof}

One might be curious about what local properties can distinguish finite type cluster algebras from more general acyclic cluster algebras, or $A$-type cluster algebras from other Dynkin type cluster algebras.  However, the theorem implies that locally acyclic cluster schemes are locally indistinguishable from isolated cluster schemes.
\begin{coro}
Any local property of a locally acyclic cluster algebra occurs in some isolated cluster algebra.
\end{coro}

\subsection{New proofs of acyclic results}

All the properties of locally acyclic cluster algebras proven so far have relied on existing proofs of the analogous results for acyclic cluster algebras.  However, by the above theorem, it would suffice to know these results only for isolated cluster algebras.  Since isolated cluster algebras are exceedingly simple, these results can be proven more easily than the general acyclic case.  Therefore, this local approach gives new proofs of these results in the acyclic case.

\begin{rem} This can be extended to other results on acyclics, as well.  For example, it is possible to prove Lemma \ref{lemma: acyclicpres} locally.  One first proves the lemma for isolated cluster algebras, which is immediate.  Then, because Lemma \ref{lemma: isolatedcover} guarantees the existence of an isolated cover which are all freezings with respect to a given seed, intersecting the algebras gives the presentation of Lemma \ref{lemma: acyclicpres}.
\end{rem}

\section{Regularity and full rank}

A local property of schemes we have not yet addressed is regularity.  This can be a delicate question even in simple cases; the cluster algebra of coefficient-free $A_n$ is singular if and only if $n\equiv 3 \, (mod\, 4)$.  However, if the exchange matrix of $\A$ is `full rank' (ie, surjective over $\mathbb{Q}$), and $\A$ is locally acyclic, then $\mathbb{Q}\otimes\A$ is a regular.  This appears to be a new result, even for acyclic cluster algebras.

\subsection{Jacobian criterion for regularity}

We recall the Jacobian criterion for regularity of a domain of dimension $d$. Let $\k$ be a perfect field, and let
\[ R:= \k[x_1,...,x_I]/\langle r_1,...,r_J\rangle\]
where the $r_1,...,r_J\in \k[x_1,...,x_I]$.  Assume that $R$ is a domain of dimension $d$.  The \textbf{Jacobian matrix} is the $J\times I$ matrix with coefficients in $R$ given by
\[\J_{ij} := \frac{\partial r_i}{\partial x_j}\]
The Jacobian matrix is an invariant attached not just to a ring, but a specific presentation of that ring.  

\begin{lemma}\cite[Theorem 16.19(b)]{Eis95}
Let $P$ be a prime ideal in $R$, and let $\phi:R\rightarrow \kappa_P$ be the map to the residue field at $P$.  The local ring $R_P$ is regular if and only if $\phi(\J)$ has rank $I-d$ as a $\kappa_P$-linear map.
\end{lemma}


\subsection{Exchange matrices}

For an ice quiver $\Q$ with vertex set $N$ and mutable vertex set $M\subset N$, the \textbf{exchange matrix} $Ex(\Q)$ is the $|M|\times |N|$ matrix with
\[ Ex(\Q)_{ij} = \Q_{ij}= \#\text{arrows from } i \text{ to } j - \#\text{arrows from } j \text{ to } i\]
The exchange matrix records most of the information from $\Q$; what is lost is the number of arrows between frozen vertices.\footnote{The reason for the distinction between $Ex(\Q)_{ij}$ and $\Q_{ij}$ is to emphasize that the rows of $Ex(\Q)$ only correspond to \emph{mutable} vertices.}

\begin{lemma}\cite[Lemma 1.2]{GSV03}
If $\Q$ and $\Q'$ are mutation-equivalent, then \[rank(Ex(\Q))=rank(Ex(\Q'))\]
\end{lemma}
Therefore, the rank of $Ex(\Q)$ is an invariant of the associated cluster algebra.  We say an ice quiver $\Q$ or cluster algebra $\A(\Q)$ has \textbf{full rank}\footnote{This is a distinct meaning of `rank' than the rank of an ice quiver or cluster algebra.  Given Proposition \ref{prop: Krulldim}, perhaps \textbf{dimension} is a better name for the rank of a cluster algebra?} if $Ex(\Q)$ is surjective as a $\mathbb{Q}$-linear map.  This class of cluster algebras is stable under freezing (and thus cluster localization).
\begin{prop}\label{prop: fullrank}
If $\Q$ has full rank, then any freezing $\Q'$ of $\Q$ has full rank.
\end{prop}
\begin{proof}
Freezing a vertex in $\Q$ amounts to discarding a row in $Ex(\Q)$.  Discarding a row cannot break surjectivity.
\end{proof}

\subsection{Regularity lemma}

The first step is a lemma which says that full rank isolated cluster algebras cannot have singularities at any point where all mutable variables vanish.

\begin{lemma}\label{lemma: isolatedsmooth}
Let $(\Q,\a)$ be an isolated seed of $\A$, and let $P$ be a prime ideal in $\mathbb{Q}\otimes\A$ containing the mutable variables $\{a_1,..,a_m\}$ in $\a$.  The exchange matrix $Ex(\Q)$ has full rank if and only if $(\mathbb{Q}\otimes\A)_P$ is a regular local ring.
\end{lemma}
\begin{proof}
As in Lemma \ref{lemma: acyclicpres}, define
\[ \pi_i^+:=\prod_{j,\Q_{ij}>0}a_j^{\Q_{ij}},\;\;\;\pi_i^-:=\prod_{j,\Q_{ij}<0}a_j^{-\Q_{ij}}\]
Then $\mathbb{Q}\otimes\A$ has a presentation
\[\mathbb{Q}\otimes \A=\mathbb{Q}[a_1,...a_m,a_{m+1}^{\pm1},...,a_n^{\pm1}, a_1',...a_m']/ \langle a_1a_1'-\pi_1^+-\pi_1^-,...a_ma_m'-\pi_m^+-\pi_m^-\rangle\]
Let $\J$ be the Jacobian matrix of this presentation

For mutable $i$, and $j$ corresponding to a frozen variable,
\[ \J_{ij} = \frac{\partial r_i}{\partial a_j} = \max(\Q_{ij},0) \frac{\pi_i^+}{a_j}-\min(\Q_{ij},0) \frac{\pi_i^-}{a_j}\]
Let $\phi:\A\rightarrow \kappa_P$ be the map to the residue field.  Since the kernel is $P$,
\[ \phi(\pi_i^++\pi_i^-)=\phi(a_ia_i')=0\]
so $\phi(\pi_i^-)=-\phi(\pi_i^+)$.  Then
\[\phi(\J_{ij}) = \max(\Q_{ij},0) \frac{\phi(\pi_i^+)}{\phi(a_j)}-\min(\Q_{ij},0) \frac{\phi(\pi_i^-)}{\phi(a_j)}=\Q_{ij}\frac{\phi(\pi_i^+)}{\phi(a_j)}\]
Since $\Q$ is isolated, $\pi_i^{\pm1}$ consist of monomials of frozen variables, and so they are invertible.  Therefore, the $m\times (n-m)$ submatrix of $\phi(\J)$ with only frozen columns is related to the submatrix of $Ex(\Q)$ with only frozen columns by a multiplying rows and columns by non-zero elements of $\kappa_P$.  

For mutable $i$, and $j$ mutable, $\phi(\J_{ij})=\delta_{ij}\phi(a_j')=0$. For mutable $i$ and $j$ one-step mutated, $\phi(\J_{ij})=\delta_{ij}\phi(a_i)=0$.  Therefore, $\phi(\J)$ is related to $Ex(\Q)$ by invertible row operations, invertible column operations and extending a matrix by zero.  Since none of these affect rank,
\[rank(\phi(\J)) = rank(Ex(\Q))\]
%

The ring $\mathbb{Q}[a_1,...a_m,a_{m+1}^{\pm1},...,a_n^{\pm1}, a_1',...a_m']$ has dimension $m+n$, and $\mathbb{Q}\otimes\A$ has dimension $n$ (Proposition \ref{prop: Krulldim}).  Therefore, by the Jacobian criterian, $(\mathbb{Q}\otimes\A)_P$ is regular if and only if $Ex(\Q)$ has full rank.
\end{proof}

\begin{rem}
The cluster algebra $\A$ may have no non-trivial ideals containing the mutable cluster variables in $\a$, and so a non-full rank cluster algebra can still be regular.
\end{rem}

\subsection{Regularity theorems}

A ring is \textbf{regular} if every local ring is regular.

\begin{lemma}
If $\A$ is an isolated cluster algebra of full rank, then $\mathbb{Q}\otimes \A$ is regular.
\end{lemma}
\begin{proof}
Consider a prime ideal $P$ in $\mathbb{Q}\otimes\A$.  Let $(\Q,\a)$ be a seed for $\A$, and let $S$ be the set of mutable cluster variables in $\a$ which are not in $P$.  By Lemma \ref{lemma: acycliclocal}, the localization of $\A$ at $S$ is a cluster localization $\A'$ of $\A$, corresponding to the freezing $(\Q',\a)$ of $(\Q,\a)$ at $S$.  By Proposition \ref{prop: fullrank}, $\A'$ is again an isolated cluster algebra of full rank, and $P$ generates a non-trivial prime ideal in $\mathbb{Q}\otimes\A'$ containing all mutable variables in $(\Q',\a)$.  By Lemma \ref{lemma: isolatedsmooth},  $(\mathbb{Q}\otimes\A')_P=(\mathbb{Q}\otimes \A)_P$ is regular.
\end{proof}
\begin{thm}\label{thm: LAsmooth}
If $\A$ is a locally acyclic cluster algebra of full rank, then $\mathbb{Q}\otimes \A$ is regular.
\end{thm}
\begin{proof}
By Theorem \ref{thm: localequiv}, $\X_\A$ may be covered by finitely many isolatd cluster charts $\{\X_{\A_i}\}$, and by Proposition \ref{prop: fullrank}, the $\{\A_i\}$ are full rank.  Changing base to $\mathbb{Q}$, $Spec(\mathbb{Q}\otimes\A)$ is covered by $\{Spec(\mathbb{Q}\otimes \A_i)\}$.  By the previous lemma, these open subschemes are regular.  Since regularity is a local property by definition, $\mathbb{Q}\otimes \A$ is regular.
%
\end{proof}
\begin{rem}
This is not a sharp criterion for the regularity of $\mathbb{Q}\otimes \A$.  The isolated cluster algebra in Remark \ref{rem: subcover} is \emph{not} full rank, but it may be covered by isolated cluster localizations which are full rank, and so it is regular.
\end{rem}
This has the following standard geometric corollary.
\begin{coro}\label{coro: LAsmooth}
If $\A$ is a locally acyclic cluster algebra of full rank, then $\X_\A(\mathbb{C}):=Hom(\A,\mathbb{C})$ is a smooth complex manifold, and $\X_\A(\R):=Hom(\A,\mathbb{R})$ is a smooth real manifold.\footnote{The $Hom$ here is in the category of ringsx.}
\end{coro}
\begin{proof}
Let $\k=\mathbb{R}$ or $\mathbb{C}$; note that 
\[\X_\A(\k)=Hom(\A,\k)=Hom(\mathbb{Q}\otimes\A,\k)\]  
For isolated $\A$, the presentation in Lemma \ref{lemma: acyclicpres} gives an embedding $\X_\A(\k)\subset \k^n$.  The Jacobian criterion then becomes the inverse function theorem over $\k$; in particular, if $\A$ is full rank, then $\X_\A(\k)$ is smooth.  For locally isolated $\A$, an isolated cover of $\X_\A$ gives an cover of $\X_\A(\k)$ by smooth $\k$-manifolds, and so it is a smooth $\k$-manifold.  (see e.g. \cite[Example 10.0.3]{Har77})
\end{proof}

\section{Failure of local acyclicity}

This section explores how a cluster algebra can fail to be locally acyclic.  The first way to show $\A$ is not locally acyclic is to show that it lacks any of the properties in Section \ref{section: properties}.  Since these can be difficult to check, we develop other techniques for showing $\A$ is not locally acyclic.

\subsection{Obstructions to coverings}

The simplest way for the Banff algorithm to fail for $\A$ is if $\A$ has no covering pairs.  This is equivalent to the following condition on quivers.
\begin{prop}
A cluster algebra $\A$ has no covering pairs if and only if, for every seed $(\Q,\a)$ of $\A$, the mutable subquiver $\Q_{mut}$ has no sources or sinks.\footnote{For the purpose of this paper, a \textbf{sink} is a vertex with at least one incoming arrow and no outgoing arrows.  A \textbf{source} is the dual notion.  Hence, an isolated vertex is neither a source nor a sink.}
\end{prop}
\begin{proof}
If $\Q_{mut}$ has no sources or sinks, then any arrow between mutable vertices in $\Q$ can be extended to a bi-infinite path.  Therefore, if $\Q_{mut}$ has no sources or sinks for every seed, it cannot have any covering pairs.  If $\Q_{mut}$ has a source $\alpha$, then there is some $\beta$ with an arrow from $\alpha$.  Since this arrow cannot be extended to a bi-infinite path, $(\alpha,\beta)$ determine a covering pair in $\A$.  Symmetrically, a sink guarentees a covering pair exists.
\end{proof}
%

There are then two fundamentally distinct ways that a cluster algebra $\A$ can fail to have a covering pair.
\begin{itemize}
\item Every seed is isolated, or
\item Every seed contains a bi-infinite path, and every arrow in $\Q_{mut}$ is in a bi-infinite path.
%
\end{itemize}
The first case corresponds to the minimal cluster charts for a locally acyclic cluster algebra.  In the second case, $\A$ has some irreducible complexity which cannot be covered by simpler cluster algebras. As we will see, these algebras are never locally acyclic.

\begin{lemma}\label{lemma: nocoveringpairs}
If there are no covering pairs for $\A$, then either $\A$ is an isolated cluster algebra, or $\A$ is not locally acyclic.
\end{lemma}
\begin{proof}
This will follow from Theorems \ref{thm: coveringhomo} and \ref{thm: nonnegnonLA}.
\end{proof}


\subsection{Non-negative and positive maps}

The absence of covering pairs may also be measured by the existence of a special algebra homomorphism.

\begin{thm}\label{thm: coveringhomo}
Let $\A$ be a cluster algebra.  Then either $\A$ admits a covering pair in some seed, or there is an algebra homomorphism $\psi:\A\rightarrow \Z$ which sends frozen variables to $1$, non-isolated mutable vertices to $0$, and isolated vertices into $\{1,2\}$.
\end{thm}
\begin{proof}
Assume that $\A$ has no covering pairs.  Then, for every seed, every mutable vertex is either isolated or contained in bi-infinite path.  For each pair of isolated cluster variables related by a mutation, choose one for $\psi$ to send to $1$, and $\psi$ sends the other to $2$.

We check that $\psi$ is well-defined, by checking that it satisfies the cluster relations.  The cluster relation at a non-isolated mutable vertex will contain three monomials, each of which contain non-isolated mutable cluster variables, so they are all sent to $0$ by $\psi$.  The cluster relation at an isolated vertex will equate the product of two isolated variables to the sum of two monomials in the frozen variables. Both sides of this are sent to $2$ by $\psi$.  Hence, $\psi$ is well-defined on $\A$.

Now, assume that $\A$ has a covering pair $(a_{i_1},a_{i_2})$ in some seed $(\Q,\a)$.  By the covering lemma, $a_{i_1}$ and $a_{i_2}$ are relatively prime, so the ideal they generate contains $1$.  In particular, no algebra homomorphism can contain both variables in the kernel, and so such a $\psi$ is impossible.
\end{proof}

This map is a specific case of a general class which will be of interest.

\begin{defn}
An algebra map $\psi:\A\rightarrow \Z$ or $\R$ is called...
\begin{itemize}
\item \textbf{positive} if, for every cluster variable $a$, $\psi(a)>0$, and
\item \textbf{non-negative} if, for every cluster variable $a$, $\psi(a)\geq0$.
\end{itemize}
%
\end{defn}

%

Positive maps are already an active area of interest in the study of cluster algebras.  The set of all positive maps can be naturally identified with the `positive part' of $\X_\A$, which is important in connections to decorated Teichm\"uller space \cite{GSV05},\cite{FG06} and total positivity \cite{FZ99}.

\begin{prop}\label{prop: totalpositive}
If a map $\psi:\A\rightarrow \Z$ or $\R$ is positive on some cluster $\a=\{a_1,..,a_n\}$, then it is positive.
%
\end{prop}
\begin{proof}
Let $a_i'$ be the mutation of $a_i$ in $\a$.  Then
\[ \psi(a_i')=\frac{\prod_{\Q_{ij}>0}\psi(a_j)^{\Q_{ij}}+\prod_{\Q_{ij}<0}\psi(a_j)^{-\Q_{ij}}}{\psi(a_i)}\]
Note the denominator is non-zero, so the expression is well-defined.  The right hand side is positive, so $\psi(a_i')>0$.  Then $\psi$ is positive on every cluster variable in the cluster $\mu_i(\a)=\{a_1,...,a_i',...,a_n\}$.  Iterating this argument, it is true for any cluster.
\end{proof}

Non-negative maps have been less studied.  As the next lemma shows, a non-negative map will either be positive or send many cluster variables to zero.
\begin{lemma}
Let $\psi$ be a non-negative map.  Either $\psi$ is positive, or in every seed $(\Q,\a)$ of $\A$, there is bi-infinite path $\{a_{i_k}\}_{k\in\mathbb{Z}}$ of cluster variables in the kernel of $\psi$.
\end{lemma}
\begin{proof}
Let $(\Q,\a)$ be any seed of $\A$.  If $\psi(\a)>0$, then Proposition \ref{prop: totalpositive} implies $\psi$ is positive.  Otherwise, there some $a_i\in \a$ with $\psi(a_i)=0$.

The variable $a_i$ cannot be a frozen variable, since $\psi(a_i^{-1})$ cannot exist.  Let $a_i'$ be the mutation of $a_i$ in $\a$.  The cluster mutation at $a_i$ gives
\[  \prod_{\Q_{ij}>0}\psi(a_j)^{\Q_{ij}}+\prod_{\Q_{ij}<0}\psi(a_j)^{-\Q_{ij}}=\psi(a_i)\psi(a_i')=0\]
Since each term on the left is $\geq0$, they both must be zero.  Therefore, there is some $a_{i+1}\in \a$ with $\psi(a_{i+1})=0$ and $A_{i,i+1}>0$, and some $a_{i-1}\in \a$ with $\psi(a_{i+1})=0$ and $\Q_{i-1,i}>0$.  Iterating this argument gives the desired directed path.
\end{proof}
\begin{coro}
If $\A$ is acyclic, every non-negative map is positive.
\end{coro}
This can be generalized to locally acyclic cluster algebras.


\begin{thm}\label{thm: nonnegnonLA}
If $\A$ is locally acyclic, every non-negative map is positive.
\end{thm}
\begin{proof}
Assume that $\A$ has a locally acyclic cover.  Let $K_\psi$ be the kernel of $\psi$; it is prime since $\Z$ and $\R$ are domains.  By assumption, there is some acyclic chart $\X_{\A'}$ which contains $K_\psi$; that is, the cluster variables which get frozen in $\A'$ are not in $K_\psi$.  By the previous lemma, each seed of $\A$ contains a direct path of cluster variables in the kernel of $\psi$.  Therefore, every seed of $\A'$ contains a directed bi-infinite path of unfrozen vertices.  In particular, every seed contains a directed cycle, this contradicts $\A'$ being acyclic.
\end{proof}
\begin{rem}\label{rem: nonnegnonLA}
The proof never uses the fact that the acyclic cover was finite.  Therefore, if $\X_\A$ has an infinite cover by acyclic cluster charts, then every non-negative map is positive.
\end{rem}
It is tempting to ask the converse.
\begin{quest}\label{quest: converse}
If $\A$ is not locally acyclic, is there always a non-negative map which is not positive?
\end{quest}
If this is true, then the converse to Remark \ref{rem: nonnegnonLA} is true: $\A$ is locally acyclic if $\X_\A$ has an infinite cover by acyclic cluster charts.  In particular, if $\X_\A=X^\circ$, then the cluster tori themselves define an infinite cover by acyclic cluster charts, and so $\A$ would be locally acyclic.

\subsection{Geometric interpretation}

Positive and non-negative maps have an important geometric interpretation.

Let $\X_\A(\R)$ denote the set of all homomorphisms $\psi:\A\rightarrow \R$.  Elements $f\in \A$ give $\mathbb{R}$-valued functions $e_f$ on $\X_\A(\R)$, by $e_f(\psi):=\psi(f)$.  The set $\X_\A(\R)$ may be made into a topological space, by equipping it with the weakest topology such that all the $e_f$ are continuous for the metric topology on $\R$.  This is called the \emph{transcendental topology} on $\X_\A(\R)$.

Each cluster torus $(\mathbb{A}_\Z^*)^n\subset \X_\A$ determines an open manifold $(\mathbb{R}^*)^n\subset \X_\A(\R)$.\footnote{These are the maps $\psi:\A\rightarrow \R$ which factor through the corresponding Laurent embedding.}  Each of these $(\mathbb{R}^*)^n$ has a connected component $\R_+^n$ on which the $e_a$ are positive.  It follows from Proposition \ref{prop: totalpositive} that this subset $\R_+^n\subset \X_\A(\R)$ is the same for each cluster torus,\footnote{Though the set is the same, the identification with $\R_+^n$ depends on a choice of cluster.} and it coincides with the set of positive maps $\psi:\A\rightarrow \R$.

By continuity, any map $\psi$ in the closure of $\R_+^n\subset\X_\A(\R)$ is non-negative.
\begin{quest}
Is every non-negative map in the closure of the set of positive maps in $\X_\A(\R)$, in the transcendental topology?
\end{quest}
\noindent This can be regarded as a question about the rigidity of non-negative maps and whether they can be deformed to positive maps.  If this question and Question \ref{quest: converse} have positive answers, then local acyclicity is equivalent to the positive part $\R_+^n$ being closed in $\X_A(\R)$.

\medskip


It is also tempting to wonder whether there is a geometric reason that all non-negative maps on a locally acyclic cluster algebra are positive.
\begin{quest}
Let $\psi$ be a non-negative map which is not positive, and let $ker(\psi)$ be its kernel, thought of as a point in $\X_\A$.  Is there always some bad local geometric behavior at $ker(\psi)\in \X_\A$ which is impossible in a locally acyclic cluster scheme?
\end{quest}
\noindent Possibilities for this bad behavior could be a singularity which is not a local complete intersection, or the Weil-Petersson form not extending to $ker(\psi)$ (see Section \ref{section: WP}).

\begin{rem}
A related question is whether the points of the form $ker(\psi)\in \X_\A$ (for $\psi$ non-negative but not positive) are always singular.  If this and the converse to Theorem \ref{thm: nonnegnonLA} are true, then $\X_\A$ being smooth implies that $\A$ is locally acyclic.  Since the cluster algebras arising in many examples are smooth, such as double Bruhat cells and Grassmannians, this would imply these are always locally acyclic.
\end{rem}

\section{Cluster algebras of marked surfaces}

An important class of cluster algebras are those coming from marked surfaces (\cite{GSV05}, \cite{FST08}, \cite{FG06}).  This section briefly reviews this construction, largely following \cite{FST08}.

%

\subsection{Marked surfaces}

A \textbf{marked surface} $(\S,\M)$ will be an oriented 2-dimensional real manifold\footnote{In contrast with \cite{FST08}, we do not assume $\S$ is connected.} $\S$, with boundary $\partial S$, together with a finite set of \textbf{marked points} $\M\subset  \S$ (some additional requirements will be imposed below).
By the classification of surfaces, any marked surface can be constructed by starting with a genus $g$ surface, cutting out a finite number of discs, and marking a finite number of points on the boundary of each disc and on the interior of $\S$.  The boundary of an excised disc will be called a \textbf{hole}, and a marked point in the interior of $\S$ will be called a \textbf{puncture}.
%

It will be necessary to exclude certain small examples, and so we assume the following of any marked surface $(\S,\M)$.
\begin{itemize}
\item $\M$ is non-empty,
\item each component of $\partial \S$ contains at least one marked point,
\item and no component of $(\S,\M)$ is one of...
\begin{itemize}
\item a sphere with at most three punctures,
\item a disc with one boundary marked point and at most one puncture, or
\item a disc with at most two boundary marked points and no punctures.\footnote{In \cite{FST08}, the authors prohibit the disc with three boundary marked points.  It will be simpler for us to allow this case, but note that it cannot have tagged arcs, and so it corresponds to the empty quiver and the cluster algebra $\Z$.}
\end{itemize}
\end{itemize}




\subsection{Tagged arcs} Inside the marked surface, we will consider arcs which connect marked points.

\begin{defn}\label{defn: arc} A \textbf{tagged arc} $\alpha$ in $(\S,\M)$ is a curve in $\S$ such that,
\begin{itemize}
\item $\alpha$ is homeomorphic to the interval $[0,1]$;
\item the endpoints are in $\M$;
\item each endpoint has a `tagging' of either \textbf{plain} (denoted
\begin{tikzpicture}[inner sep=0.5mm,scale=.5]
	\node (1) at (0,0) [marked] {};
	\draw (1,0) to [-] (1);
\end{tikzpicture}
)
or \textbf{notched} (denoted
\begin{tikzpicture}[inner sep=0.5mm,scale=.5]
	\node (1) at (0,0) [marked] {};
	\draw (1,0) to [-hooks reversed] (1);
\end{tikzpicture}
);
\item $\alpha$ has no self-intersections, except the endpoints may coincide;
\item if the endpoints coincide, they are either both plain or both notched;
\item other than the endpoints, $\alpha$ is disjoint from $\M$;
\item any endpoint in $\partial\S$ is plain; and
\item $\alpha$ does not cut out...
\begin{itemize}
\item a disc with one boundary marked point and at most one puncture
, or
\item a disc with two boundary marked points and no punctures.
\end{itemize}
\end{itemize}
\end{defn}
Tagged arcs are defined up to endpoint-fixed homotopy.  That is, two tagged arcs $\alpha$ and $\beta$ are \textbf{equivalent} if there is a continuously-varying family of arcs from $\alpha$ to $\beta$ (where the taggings are constant).  Arcs should then be thought of as representing flexible strands connecting their endpoints.  Furthermore, they have no prefered orientation.

\subsection{Triangulations} A pair of tagged arcs $\alpha$ and $\beta$ are \textbf{compatible} if
\begin{itemize}
\item $\alpha$ and $\beta$ do not intersect (except possibly at the endpoints);
\item if the untagged versions of $\alpha$ and $\beta$ are different, and $\alpha$ and $\beta$ share an endpoint, then the taggings at that endpoint agree;
\item if the untagged versions of $\alpha$ and $\beta$ coincide, then at least one endpoint of $\alpha$ must have the same tagging as $\beta$.
\end{itemize}
A collection of tagged arcs $\Delta$ is \textbf{compatible} if it is pairwise compatible.


\begin{defn}\label{def: tri}
A \textbf{tagged triangulation} $\Delta$ of $(\S,\M)$ is a maximal compatible collection of distinct tagged arcs.
\end{defn}

Given a tagged triangulation $\Delta$ and $\alpha\in \Delta$, there is exactly one other tagged arc $\beta$ such that $(\Delta-\{\alpha\})\cup \{\beta\}$ is also a tagged triangulation.  This is arguably the purpose behind adding these taggings and the many rules they satisfy.

\begin{rem}
A triangulation of $(\S,\M)$ gives a tagged triangulation (where all taggings are plain) exactly when there are no triangles with two edges identified.
\end{rem}

\begin{rem}
If $(\S,\M)$ has no punctures, then every endpoint of a tagged arc must be plain.  In this case, `tagged arcs' coincide with regular arcs, and `tagged triangulations' coincide with regular triangulations.
\end{rem}

%

The number of tagged arcs in a tagged triangulation can be computed from the genus $g$, the number of punctures $p$, the total number of marked points $|\M|$, and the number of holes $h$.
\begin{prop}\cite[Proposition 2.10]{FST08}
Every tagged triangulation of $(\S,\M)$ is finite, with total number of tagged arcs given by
\[ 6g+3h+2p +|\M|-6\]
\end{prop}

\subsection{Cluster algebras of marked surfaces.}  Let $(\S,\M)$ be a marked surface, and let $\Delta$ be a tagged triangulation.  For simplicity, if $\alpha,\beta\in \Delta$ are homotopic, we will treat them as if they are identical (except for taggings). We construct a quiver $\Q_\Delta$ as follows.  Notice that the taggings are ignored for the construction of $\Q_\Delta$.
\begin{itemize}
\item $\Q_\Delta$ has a vertex for each tagged arc in $\Delta$.
\item At each marked point $m$ (which is not a puncture contained in two tagged arcs in $\Delta$), add an arrow from $\alpha$ to $\beta$ whenever $\beta$ appears immediately after $\alpha$ when going counter-clockwise around $m$.\footnote{If $\beta'$ is homotopic to $\beta$ except for taggings, there will also be an arrow from $\alpha$ to $\beta'$.  Similarly, if $\alpha'$ is homotopic to $\alpha$ except for taggings, there will also be an arrow from $\alpha'$ to $\beta$ (and $\alpha'$ to $\beta'$).}
\end{itemize}
%
\noindent This can be regarded as an ice quiver with no frozen vertices.  As defined, $\Q_\Delta$ is automatically loop and 2-cycle free.
\begin{rem}
Punctures in two tagged arcs are ignored for simplicity.  A puncture in two arcs would create two opposite arrows which would cancel each other.
\end{rem}
\begin{ex}
Let $(\S,\M)$ be the disc with two marked points on the boundary, and three punctures.  
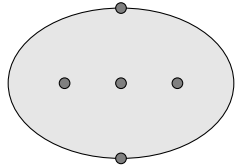
\begin{figure}[h!]
\centering
\begin{tikzpicture}[inner sep=0.5mm,xscale=1.5,scale=.5,auto]
	\draw[fill=black!10] (0,0) circle (2);
	\node (1) at (0,2) [circle,draw,fill=black!50] {};
	\node (2) at (0,-2) [circle,draw,fill=black!50] {};
	\node (4) at (-1,0) [circle,draw,fill=black!50] {};
	\node (5) at (0,0) [circle,draw,fill=black!50] {};
	\node (6) at (1,0) [circle,draw,fill=black!50] {};
\end{tikzpicture}
\end{figure}

Three tagged triangulations and their quivers are shown in Figure \ref{fig: triang}.
\begin{figure}[h!]
\centering
\begin{tikzpicture}
\begin{scope}
\begin{scope}[xshift=-1.5in,xscale=1.5,inner sep=0.5mm,scale=.5,auto]
	\draw[fill=black!10] (0,0) circle (2);
	\node (1) at (0,2) [circle,draw,fill=black!50] {};
	\node (2) at (0,-2) [circle,draw,fill=black!50] {};
	\node (4) at (-1.2,0) [circle,draw,fill=black!50] {};
	\node (5) at (0,0) [circle,draw,fill=black!50] {};
	\node (6) at (1.2,0) [circle,draw,fill=black!50] {};
	\draw (1) to [-] (4);
	\draw (1) to [-] (5);
	\draw (1) to [-] (6);
	\draw (2) to [-] (4);
	\draw (2) to [-] (5);
	\draw (2) to [-] (6);
	\draw (4) to (5);
	\draw (5) to (6);
\end{scope}
\begin{scope}[xshift=0in,xscale=1.5,inner sep=0.5mm,scale=.5,auto]
	\draw[fill=black!10] (0,0) circle (2);
	\node (1) at (0,2) [circle,draw,fill=black!50] {};
	\node (2) at (0,-2) [circle,draw,fill=black!50] {};
	\node (4) at (-1.2,0) [circle,draw,fill=black!50] {};
	\node (5) at (0,0) [circle,draw,fill=black!50] {};
	\node (6) at (1.2,0) [circle,draw,fill=black!50] {};
	\draw (1) to [-] (4);
	\draw (1) to [-] (5);
	\draw (1) to [-] (6);
	\draw (2) to [-] (4);
	\draw (2) to [-] (5);
	\draw (2) to [-] (6);
	\draw (1) to [out=290,in=70] (2);
	\draw (1) to [out=250,in=110] (2);
\end{scope}
\begin{scope}[xshift=1.5in,xscale=1.5,inner sep=0.5mm,scale=.5,auto]
	\draw[fill=black!10] (0,0) circle (2);
	\node (1) at (0,2) [circle,draw,fill=black!50] {};
	\node (2) at (0,-2) [circle,draw,fill=black!50] {};
	\node (4) at (-1.2,0) [circle,draw,fill=black!50] {};
	\node (5) at (0,0) [circle,draw,fill=black!50] {};
	\node (6) at (1.2,0) [circle,draw,fill=black!50] {};
  \draw (2) to [out=30,in=135,relative,-hooks reversed] (4);
	\draw (1) to [-] (5);
	\draw (2) to [out=-30,in=225,relative,-hooks reversed] (6);
	\draw (2) to [-] (4);
	\draw (1) to [out=15,in=135,relative,-hooks reversed] (5);
	\draw (2) to [-] (6);
	\draw (1) to [out=45,in=165,relative] (2);
	\draw (1) to [out=250,in=110] (2);
\end{scope}
\end{scope}
\vspace{.3in}
\begin{scope}
\begin{scope}[xshift=-1.5in,yshift=-1in,inner sep=0.5mm,scale=.5,auto]
    \node (1) at (-2,1) [mutable] {};
    \node (2) at (-2,-1) [mutable] {};
    \node (3) at (-1,0) [mutable] {};
    \node (4) at (0,1) [mutable] {};
    \node (5) at (0,-1) [mutable] {};
    \node (6) at (1,0) [mutable] {};
    \node (7) at (2,1) [mutable] {};
    \node (8) at (2,-1) [mutable] {};

    \draw (1) to [-angle 90] (4);
    \draw (4) to [-angle 90] (7);
    \draw (8) to [-angle 90] (5);
    \draw (5) to [-angle 90] (2);
    \draw (1) to [-angle 90] (2);
    \draw (2) to [-angle 90] (3);
    \draw (3) to [-angle 90] (1);
    \draw (8) to [-angle 90] (7);
    \draw (7) to [-angle 90] (6);
    \draw (6) to [-angle 90] (8);
    \draw (3) to [-angle 90] (5);
    \draw (5) to [-angle 90] (6);
    \draw (6) to [-angle 90] (4);
    \draw (4) to [-angle 90] (3);
\end{scope}
\begin{scope}[xshift=0in,yshift=-1in,inner sep=0.5mm,scale=.5,auto]
    \node (1) at (-2,1) [mutable] {};
    \node (2) at (-2,-1) [mutable] {};
    \node (3) at (-1,0) [mutable] {};
    \node (4) at (0,1) [mutable] {};
    \node (5) at (0,-1) [mutable] {};
    \node (6) at (1,0) [mutable] {};
    \node (7) at (2,1) [mutable] {};
    \node (8) at (2,-1) [mutable] {};

    \draw (1) to [-angle 90] (3);
    \draw (3) to [-angle 90] (2);
    \draw (3) to [-angle 90] (4);
    \draw (4) to [-angle 90] (6);
    \draw (6) to [-angle 90] (5);
    \draw (5) to [-angle 90] (3);
    \draw (8) to [-angle 90] (6);
    \draw (6) to [-angle 90] (7);
\end{scope}
\begin{scope}[xshift=1.5in,yshift=-1in,inner sep=0.5mm,scale=.5,auto]
    \node (1) at (-2,1) [mutable] {};
    \node (2) at (-2,-1) [mutable] {};
    \node (3) at (-1,0) [mutable] {};
    \node (4) at (0,1) [mutable] {};
    \node (5) at (0,-1) [mutable] {};
    \node (6) at (1,0) [mutable] {};
    \node (7) at (2,1) [mutable] {};
    \node (8) at (2,-1) [mutable] {};

    \draw (1) to [-angle 90] (3);
    \draw (2) to [-angle 90] (3);
    \draw (4) to [-angle 90] (3);
    \draw (6) to [-angle 90] (4);
    \draw (6) to [-angle 90] (5);
    \draw (5) to [-angle 90] (3);
    \draw (6) to [-angle 90] (8);
    \draw (6) to [-angle 90] (7);
    \draw (3) to [-angle 90] (6);
\end{scope}
\end{scope}
\end{tikzpicture}
\caption{Sample triangulations and their quivers}
\label{fig: triang}
\end{figure}
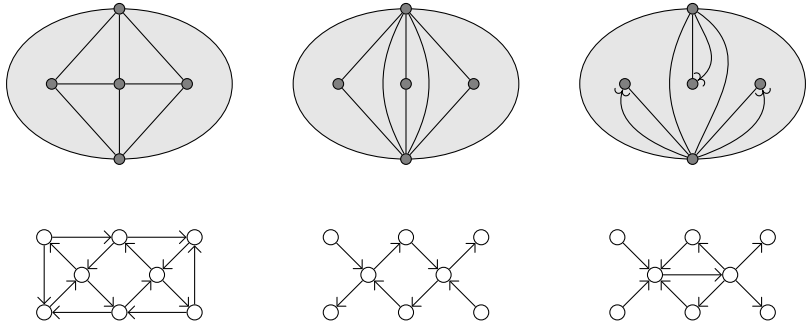
\end{ex}

The quivers coming from tagged triangulations can be used to defined cluster algebras because of the following lemma.
\begin{lemma}\cite[Proposition 4.10]{FST08}
Let $\Delta$ and $\Delta'$ be two tagged triangulations of $(\S,\M)$.  Then there is a canonical isomorphism
\[ \A(\Q_{\Delta})\simeq \A(\Q_{\Delta'})\]
\end{lemma}
Define the \textbf{cluster algebra} of $(\S,\M)$ by
\[\A(\S,\M):=\A(\Q_{\Delta}),\] where $\Delta$ is any triangulation  of $(\S,\M)$.  By the lemma, this is independent of the choice of tagged triangulation.
\begin{thm}\cite[Theorems 5.6 and 7.11]{FST08},\cite[Theorem 5.1]{FST??}
\begin{itemize}
\item If $\M$ does not consist of one or two punctures, the cluster variables in $\A(\S,\M)$ are canonically identified with tagged arcs in $(\S,\M)$; and seeds in $\A(\S,\M)$ are identified with tagged triangulations $\Delta$ with quiver of the form $\Q_\Delta$.
\item If $\M$ consists of exactly one puncture,  the cluster variables in $\A(\S,\M)$ are canonically identified with plain tagged arcs in $(\S,\M)$; and seeds in $\A(\S,\M)$ are identified with plain tagged triangulations $\Delta$ with quiver of the form $\Q_\Delta$.
\end{itemize}
\end{thm}

\subsection{Cutting marked surfaces} Let $\alpha$ be a tagged arc in $(\S,\M)$.  The \textbf{cutting} of $(\S,\M)$ along $\alpha$ is the marked surface obtained by cutting $\S$ along $\alpha$, compactifying $\S$ by adding boundary along the two sides of $\alpha$, and adding marked points where the endpoints of $\alpha$ were.  If $(\S',\M')$ is the cutting along $\alpha$, then there is a natural map
\[ (\S',\M')\rightarrow (\S,\M)\]
which is a bijection away from $\alpha\subset \S$, a 2-to-1 map over the interior of $\alpha$, and such that the preimage of $\M$ is $\M'$.  The three types of cut are pictured in Figure \ref{fig: cuts}, but note that the endpoints of $\alpha$ need not be distinct.

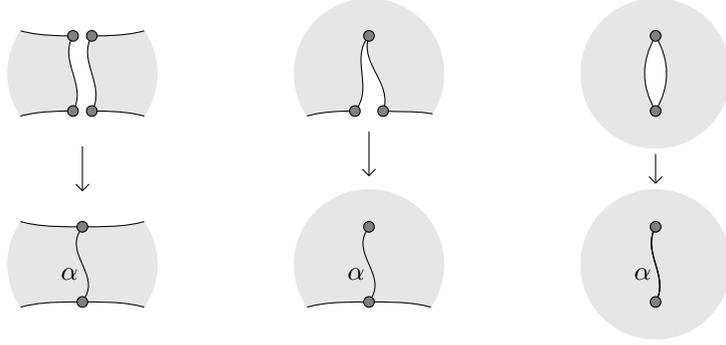
\begin{figure}[h!]
\centering
\begin{tikzpicture}[inner sep=0.5mm,auto]
\begin{scope}[xshift=-1.5in]
\begin{scope}[yshift=.5in,scale=.25]
	\clip (0,0) circle (4);
	\draw[fill=black!10] (-5,-3) to [in=180,out=30] (-.5,-2) to [in=240,out=60] (-.5,2) to [out=180,in=-30] (-5,3);
	\draw[fill=black!10] (5,3) to [in=0,out=210] (.5,2) to [in=60,out=240] (.5,-2) to [out=0,in=150] (5,-3);
	\node (1) at (-.5,-2) [marked] {};
	\node (2) at (-.5,2) [marked] {};
	\node (3) at (.5,-2) [marked] {};
	\node (4) at (.5,2) [marked] {};
\end{scope}
	\draw (0,.3) to [-angle 90] (0,-.3);
\begin{scope}[yshift=-.5in,scale=.25]
	\clip (0,0) circle (4);
	\draw[fill=black!10] (-5,-3) to [in=180,out=30] (0,-2) to [in=150,out=0] (5,-3) to [line to] (5,3) to [in=0,out=210] (0,2) to [in=-30,out=180] (-5,3);
	\node (1) at (0,-2) [marked] {};
	\node (2) at (0,2) [marked] {};
	\draw (1) to [in=240,out=60] node {$\alpha$} (2);
\end{scope}
\end{scope}
\begin{scope}[xshift=0in]
\begin{scope}[yshift=.5in,scale=.25]
	\clip (0,0) circle (4);
	\draw[fill=black!10] (-5,-3) to [in=180,out=30] (-.75,-2) to [in=150,out=-30,relative] (0,2) to [in=150,out=-30] (.75,-2) to [in=150,out=10] (5,-3) to [line to] (5,5) to (0,5) to (-5,5);
	\node (1) at (-.75,-2) [marked] {};
	\node (2) at (.75,-2) [marked] {};
	\node (3) at (0,2) [marked] {};
\end{scope}
	\draw (0,.5) to [-angle 90] (0,-.1);
\begin{scope}[yshift=-.5in,scale=.25]
	\clip (0,0) circle (4);
	\draw[fill=black!10] (-5,-3) to [in=180,out=30] (0,-2) to [in=150,out=0] (5,-3) to [line to] (5,5) to (0,5) to (-5,5);
	\node (1) at (0,-2) [marked] {};
	\node (2) at (0,2) [marked] {};
	\draw (1) to [in=240,out=60] node {$\alpha$} (2);
\end{scope}
\end{scope}
\begin{scope}[xshift=1.5in]
\begin{scope}[yshift=.5in,scale=.25]
	\clip (0,0) circle (4);
    \draw[fill=black!10] (-5,-5) to (-5,5) to (5,5) to (5,-5);
	\draw[fill=black!0] (0,-2) to [in=150,out=30,relative] (0,2) to [in=150,out=30] (0,-2);
	\node (1) at (0,-2) [marked] {};
	\node (2) at (0,2) [marked] {};
\end{scope}
	\draw (0,.2) to [-angle 90] (0,-.2);
\begin{scope}[yshift=-.5in,scale=.25]
	\clip (0,0) circle (4);
    \draw[fill=black!10] (-5,-5) to (-5,5) to (5,5) to (5,-5);
	\draw[fill=black!0] (0,-2) to [in=150,out=-30,relative] node {$\alpha$} (0,2) to [in=150,out=-30] (0,-2);
	\node (1) at (0,-2) [marked] {};
	\node (2) at (0,2) [marked] {};
\end{scope}
\end{scope}
\end{tikzpicture}
\caption{Types of cuts}
\label{fig: cuts}
\end{figure}

Let $\Delta$ be a tagged triangulation of $(\S,\M)$, with $\alpha\in \Delta$.  If $(\S',\M')$ is the cutting of $(\S,\M)$ along $\alpha$, then $\Delta$ determines a tagged triangulation $\Delta'$ of $(\S',\M')$ as follows.
\begin{itemize}
\item If there is a tagged arc $\alpha'\in \Delta$ which is identical to $\alpha$ except for the taggings, then there is exactly one endpoint $m$ of $\alpha$ where the taggings of $\alpha$ and $\alpha'$ match. Then $\Delta'$ is the preimage of $\Delta-\{\alpha,\alpha'\}$ in $(\S',\M')$ together with a new tagged arc connecting the two preimages of $m$.
\item If there is no such $\alpha'$, then $\Delta'$ is the preimage of $\Delta-\{\alpha\}$.
\end{itemize}
In each case, any notched ends of a tagged arc which touch the boundary are changed to plain.

Cutting has a very simple effect on the corresponding quivers, it deletes the vertex corresponding to $\alpha$.
\begin{lemma}\label{lemma: cutting}
Let $\Delta$ and $\Delta'$ be as above.  Then $\Q_{\Delta'}$ is the induced subquiver of $\Q_\Delta$ on the complement of the vertex corresponding to $\alpha$.
\end{lemma}

\section{Marked surfaces and local acyclicity}

We turn to the question of when cluster algebras of marked surfaces are locally acyclic.  The main result will be that a marked surface with `enough' boundary marked points (at least two per component of $\S$) defines a locally acyclic cluster algebra.

\subsection{Covering pairs for marked surfaces}

To produce acyclic covers of cluster algebras of marked surfaces, we will need to be able to produce covering pairs.  The following two lemmas give simple topological methods for producing covering pairs.

\begin{lemma}\label{lemma: type1}
Let $\alpha$ and $\beta$ be tagged arcs in $(\S,\M)$ which cut out an unpunctured triangle (Figure \ref{fig: typeA}), with all endpoints in $\partial\S$ and the endpoints of $\beta$ are distinct.  Then $(\alpha,\beta)$ are a covering pair in $\A(\S,\M)$.
\end{lemma}
\begin{figure}[h!]
\centering
\begin{tikzpicture}[inner sep=0.5mm,auto,scale=.75]
	\useasboundingbox (-2.5,-1.4) rectangle (2.5,1);
\begin{scope}[scale=.4]
\begin{scope}
	\clip (0,0) circle (4);
	\draw[fill=black!10] (-5,-3) to [in=190,out=30] (-2,-2) to [in=170,out=10] (1,-2) to [in=150,out=-10] (5,-3) to [line to] (5,3) to [in=10,out=210] (2,2) to [in=-10,out=190] (-1,2) to [in=-30,out=170] (-5,3);
	\node (1) at (1,-2) [marked] {};
	\node (2) at (-1,2) [marked] {};
	\node (3) at (-2,-2) [marked] {};
	\draw (1) to [in=150,out=-30,relative] node {$\beta$} (2);
	\draw (2) to [in=150,out=-30,relative] node[swap] {$\alpha$} (3);
\end{scope}
\end{scope}
\end{tikzpicture}
\caption{A covering pair (Type A)}
\label{fig: typeA}
\end{figure}
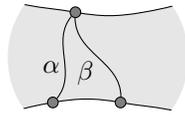
\begin{proof}
Since the endpoints of $\beta$ are distinct, we may consider the quadralateral with diagonal $\beta$ pictured below.
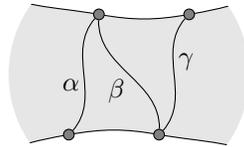
\begin{center}
\begin{tikzpicture}[inner sep=0.5mm,auto]
	\useasboundingbox (-2.5,-.8) rectangle (2.5,1);
\begin{scope}[scale=.4]
\begin{scope}
	\clip (0,0) circle (4);
	\draw[fill=black!10] (-5,-3) to [in=190,out=30] (-2,-2) to [in=170,out=10] (1,-2) to [in=150,out=-10] (5,-3) to [line to] (5,3) to [in=10,out=210] (2,2) to [in=-10,out=190] (-1,2) to [in=-30,out=170] (-5,3);
	\node (1) at (1,-2) [marked] {};
	\node (2) at (-1,2) [marked] {};
	\node (3) at (-2,-2) [marked] {};
	\node (4) at (2,2) [marked] {};
	\draw (1) to [in=150,out=-30,relative] node {$\beta$} (2);
	\draw (2) to [in=150,out=-30,relative] node[swap] {$\alpha$} (3);
	\draw (1) to [in=150,out=-30,relative] node[swap] {$\gamma$} (4);
\end{scope}
\end{scope}
\end{tikzpicture}
\end{center}
Despite the picture, it is possible that $\alpha$ and $\gamma$ coincide or share endpoints.  If $\gamma$ is a boundary segment, discard it.  If $\gamma$ cuts out a once-punctured monogon, replace it with the pair of tagged arcs pictured below.
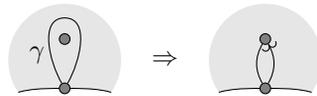
\begin{center}
\begin{tikzpicture}[inner sep=0.5mm,auto,scale=.75]
	\useasboundingbox (-2.5,-.6) rectangle (2.5,1);
\begin{scope}[xshift=-.7in,scale=.25]
	\clip (0,0) circle (4);
	\draw[fill=black!10] (-5,-3) to [in=180,out=30] (0,-2) to [in=150,out=0] (5,-3) to [line to] (5,5) to (0,5) to (-5,5);
	\node (1) at (0,-2) [marked] {};
	\node (2) at (0,1.5) [marked] {};
	\draw (1) to [in=180,out=120] node{$\gamma$} (0,3) to [in=60,out=0] (1);
\end{scope}
	\node (=>) at (0,0) {$\Rightarrow$};
\begin{scope}[xshift=.7in,scale=.25]
	\clip (0,0) circle (4);
	\draw[fill=black!10] (-5,-3) to [in=180,out=30] (0,-2) to [in=150,out=0] (5,-3) to [line to] (5,5) to (0,5) to (-5,5);
	\node (1) at (0,-2) [marked] {};
	\node (2) at (0,1.5) [marked] {};
	\draw (1) to [in=240,out=120]  (2);
	\draw (1) to [in=300,out=60,-hooks reversed]  (2);
\end{scope}
\end{tikzpicture}
\end{center}

Choose any triangulation $\Delta$ of $(\S,\M)$ which contains $\alpha,\beta$ and $\gamma$.  There is an arrow from $\alpha$ to $\beta$ in $\Q_\Delta$.  At each end of $\beta$, there will be no tagged arcs in $\Delta$ which are counter-clockwise to $\beta$, and so there are no arrows out of $\beta$ in $\Q_\Delta$.
\end{proof}

\begin{lemma}\label{lemma: type2}
Let $\alpha$ be a tagged arc in $(\S,\M)$ which cuts out a once-punctured digon with radius $\beta$ (Figure \ref{fig: typeB}).\footnote{The endpoints of $\alpha$ need not be distinct.}  Then $(\alpha,\beta)$ are a covering pair in $\A(\S,\M)$.
\end{lemma}
\begin{figure}[h!]
\centering
\begin{tikzpicture}[inner sep=0.5mm,auto]
	\useasboundingbox (-2.5,-.8) rectangle (2.5,1.1);
\begin{scope}[scale=.25]
\begin{scope}[rotate=90]
	\clip (-.5,0) circle (5);
	\draw[fill=black!10] (3,0) circle (5);
	\node (1) at (0,4) [marked] {};
	\node (2) at (0,-4) [marked] {};
	\node (3) at (0,0) [marked] {};
	\draw (2) to [in=165,out=-15,relative] node[swap] {$\beta$} (3);
	\draw (1) to [in=127,out=53,relative] node {$\alpha$} (2);
\end{scope}
\end{scope}
\end{tikzpicture}
\caption{A covering pair (Type B)}
\label{fig: typeB}
\end{figure}
\begin{proof}
Let $\gamma$ be the other plain tagged radius of the digon.
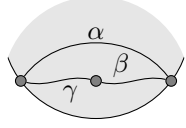
\begin{center}
\begin{tikzpicture}[inner sep=0.5mm,auto]
	\useasboundingbox (-2.5,-.5) rectangle (2.5,1.2);
\begin{scope}[scale=.25]
\begin{scope}[rotate=90]
	\clip (-.5,0) circle (5);
	\draw[fill=black!10] (3,0) circle (5);
	\node (1) at (0,4) [marked] {};
	\node (2) at (0,-4) [marked] {};
	\node (3) at (0,0) [marked] {};
	\draw (1) to [in=165,out=-15,relative] node[swap] {$\gamma$} (3);
	\draw (1) to [in=127,out=53,relative] node {$\alpha$} (2);
	\draw (2) to [in=165,out=-15,relative] node[swap] {$\beta$} (3);
\end{scope}
\end{scope}
\end{tikzpicture}
\end{center}
Choose any triangulation $\Delta$ of $(\S,\M)$ which contains $\alpha,\beta$ and $\gamma$.  There is an arrow from $\alpha$ to $\beta$ in $\Q_\Delta$.  At each end of $\beta$, there will be no tagged arcs in $\Delta$ which are counter-clockwise to $\beta$, and so there are no arrows out of $\beta$ in $\Q_\Delta$.
\end{proof}

\subsection{Local acyclicity}

We now prove several classes of cluster algebra of a marked surface are locally acyclic by a topological version of the Banff algorithm.
\begin{lemma}\label{lemma: surfacecutting}
Let $(\alpha,\beta)$ be a covering pair in $\A(\S,\M)$, and let $(\S_\alpha,\M_\alpha)$ and $(\S_\beta,\M_\beta)$ be the cuttings along $\alpha$ and $\beta$, respectively.  If $\A(\S_\alpha,\M_\alpha)$ and $\A(\S_\beta,\M_\beta)$ are locally acyclic, then so is $\A(\S,\M)$.
\end{lemma}
\begin{proof}
By Lemma \ref{lemma: cutting}, cutting along $\alpha$ or $\beta$ amounts to deleting the corresponding vertices in $\Q_{\Delta}$.  If $\A(\S_\alpha,\M_\alpha)$ and $\A(\S_\beta,\M_\beta)$ are locally acyclic, then by Proposition \ref{prop: rBanff}, $\A(\S,\M)$ is locally acyclic.
%
%
%
\end{proof}

We now prove a succession of more general classes of surface are locally acyclic.  Each will be by induction on the rank of a cluster algebra, which here is the size of any triangulation.  Cutting a marked surface drops its rank by 1, by Lemma \ref{lemma: cutting}.

\begin{lemma}
If $\S$ is a disc, then $\A(\S,\M)$ is locally acyclic.
\end{lemma}
\begin{proof}
The two basic cases are the unpunctured triangle, which is the empty cluster algebra; and the once-punctured digon, which is the cluster algebra on two isolated vertices.  Both are locally acyclic.  Now, assume that $\A$ is locally acyclic for every marked disc of rank strictly less than $n$, and let $(\S,\M)$ be a marked disc of rank $n$.

If $(\S,\M)$ has any punctures, it will have a covering pair $(\alpha,\beta)$ coming from Lemma \ref{lemma: type2}.  The cutting of $(\S,\M)$ along $\alpha$ is the union of a once-punctured digon with a disc of smaller rank, and so its cluster algebra is locally acyclic.  The cutting of $(\S,\M)$ along $\beta$ is a disc of smaller rank, and so its cluster algebra is locally acyclic.  Therefore, by Lemma \ref{lemma: surfacecutting}, $\A(\S,M)$ is locally acyclic.

If $(\S,\M)$ has no punctures, $\A(\S,\M)$ is of type $A_n$ and is acyclic.  
\end{proof}

\begin{thm}\label{thm: inadisc}
If $\S$ includes into a disc, then $\A(\S,\M)$ is locally acyclic.
\end{thm}
\begin{proof}
If the rank is small enough, $(\S,\M)$ must be a disc, and is locally acyclic.  Now, assume the theorem is true for $\A(S,\M)$ of rank strictly less than $n$, and let $(\S,\M)$ be a marked surface in a disc of rank $n$.

If $(\S,\M)$ is not connected, then it is the union of two marked surfaces of smaller rank which include into a disc, and so it is locally acyclic by the inductive hypothesis.  If $(\S,\M)$ is connected by only has one boundary component, it is a disc and is covered by the previous lemma.  If $(\S,\M)$ is connected and has more than one boundary component, let $\alpha$ be a tagged arc which connects distinct boundary components, and find a covering pair $(\alpha,\beta)$ as in Lemma \ref{lemma: type1}.  The cutting along $\alpha$ or $\beta$ will each be a marked surface of strictly smaller rank which includes into a disc, and so they are locally acyclic by the inductive hypothesis.
\end{proof}

\begin{thm}\label{thm: atleast2}
If $(\S,\M)$ has at least two boundary marked points in each component of $\S$, then $\A(\S,\M)$ is locally acyclic.
\end{thm}
\begin{proof}
Assume the theorem is true for $\A(S,\M)$ of rank strictly less than $n$, and let $(\S,\M)$ be a marked surface of rank $n$ with at least two boundary marked points in each component. 

If $(\S,\M)$ is not connected, then each component is a marked surface of strictly smaller rank, and so $\A(\S,\M)$ is locally acyclic by the inductive hypothesis.

If $(\S,\M)$ is connected, and $\partial\S$ has at least two components, then there is some $\beta$ connecting marked points on distinct components.  Find an arc $\alpha$ as in Lemma \ref{lemma: type1}, so that $(\alpha,\beta)$ is a covering pair.  Then the cutting of $(\S,\M)$ along $\alpha$ or $\beta$ is a connected marked surface with at least two marked boundary points, of rank $<n$.  Then $\A(\S,\M)$ is locally acyclic.

If $(\S,\M)$ is connected, $\partial\S$ has a component with at least two marked points, and the genus of $\S$ is not zero, then there is some arc $\beta$ which connects distinct marked points on the same component of $\partial\S$, such that the cutting along $\beta$ does not disconnect $\S$.  Find an arc $\alpha$ as in Lemma \ref{lemma: type1}, so that $(\alpha,\beta)$ is a covering pair.  Then the cutting of $(\S,\M)$ along $\alpha$ or $\beta$ is again a marked surface with at least two marked boundary points in each component, of smaller rank.  Then $\A(\S,\M)$ is locally acyclic.

If the genus of $\S$ is zero, then $\S$ includes into a disc and $\A(\S,\M)$ is locally acyclic by a previous theorem.
\end{proof}

\begin{rem}
Compare Theorem \ref{thm: atleast2} with the following theorem.
\begin{thm}\cite[Corollary 12.4]{FST08}
The cluster algebra $\A(\S,\M)$ is acyclic if and only if it has a seed of finite-type or affine-type.
\end{thm}
\noindent Since most cluster algebras of marked surfaces are not finite-type or affine-type, marked surfaces give a large class of examples of locally acyclic cluster algebras which are not acyclic.
\end{rem}

\subsection{Marked surfaces and failure of local acyclicity} In this section, we show that some marked surfaces give cluster algebras which are not locally acyclic.

%
%

\begin{thm}\label{thm: noboundary}
If $\partial\S$ is empty, then $\A(\S,\M)$ has no covering pairs and is not locally acyclic.
\end{thm}
\begin{proof}
Let $\Delta$ be any triangulation of $(\S,\M)$, and $\alpha\in \Delta$.  Let $m$ be the endpoint of $\alpha$ which is in at least three tagged arcs.  There will be tagged arcs immediately clockwise and counter-clockwise of $\alpha$, and so in $\Q_\Delta$ there are arrows in and out of $\alpha$.  Therefore, there are no sources and sinks in any seed of $\A(\S,\M)$, and so there are no covering pairs.  By Lemma \ref{lemma: nocoveringpairs}, $\A(\S,\M)$ is not locally acyclic.
\end{proof}

\begin{thm}\label{thm: oneboundary}
If $|\M|=1$, then $\A(\S,\M)$ has no covering pairs and is not locally acyclic.
\end{thm}
\begin{proof}
Let $\Delta$ be any triangulation of $(\S,\M)$, and $\alpha\in \Delta$.  Both ends of $\alpha$ are at the unique endpoint, so at most one of them can be the most counter-clockwise arc coming into that point.  Then there is at least one arrow coming out of $\alpha$ in $\Q_\Delta$.  Similarly, there is at least one arrow coming into $\alpha$.  Therefore, there are no sources and sinks in any seed of $\A(\S,\M)$, and so there are no covering pairs.  By Lemma \ref{lemma: nocoveringpairs}, $\A(\S,\M)$ is not locally acyclic.
\end{proof}

\begin{rem}\label{rem: nonA}
Together with the results of the previous section, the only unknown case is when $(\S,\M)$ has a component with 
\begin{itemize}
\item $\partial\S$ connected with a single marked point,
\item there is at least one puncture, and
\item $\S$ has positive genus.  
\end{itemize}
Here, covering pairs exist, but the Banff algorithm can be shown to always fail.
\end{rem}

%

%
%
%
%

\section{Examples and non-examples}\label{section: examples}

\subsection{A small example}

Let $\Q$ be the quiver in Figure \ref{fig: smallex}, and $\A=\A(\Q)$.\footnote{This example is due to Jesse Levitt, a student in LSU's Cluster Algebras seminar.}
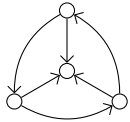
\begin{figure}[h!]
\centering
\begin{tikzpicture}[scale=.7,auto,minimum size=2mm]
    \node (1) at (-1,-.58) [mutable] {};
    \node (2) at (1,-.58) [mutable] {};
    \node (3) at (0,1.15) [mutable] {};
    \node (4) at (0,0) [mutable] {};

    \draw (1) to [out=330,in=210,-angle 90] (2);
    \draw (2) to [out=90,in=330,-angle 90] (3);
    \draw (3) to [out=210,in=90,-angle 90] (1);
    \draw (1) to [-angle 90] (4);
    \draw (2) to [-angle 90] (4);
    \draw (3) to [-angle 90] (4);
\end{tikzpicture}
\caption{A simple local acyclic}
\label{fig: smallex}
\end{figure}

The cluster algebra $\A$ is \textbf{locally acyclic} by Figure \ref{fig: Banff1} and Theorem \ref{thm: Banff}.  This quiver is likely not acyclic; however, this is difficult to check since this quiver is equivalent to infinitely many other quivers.  One can check, by hand or computer,\footnote{The applet at \cite{QuiverApplet} is highly recommended.} that there are no acyclic quivers within eight mutations.  Therefore is likely the smallest example of a locally acyclic cluster algebra which is not acyclic.

\subsection{The Markov cluster algebra}\label{section: Markov}

Let $\Q$ be the quiver in Figure \ref{fig: Markov}, and $\A=\A(\Q)$ be the \textbf{Markov cluster algebra}.
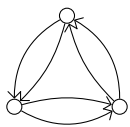
\begin{figure}[h!]
\centering
\begin{tikzpicture}[scale=.7,auto,minimum size=2mm]
	\node (1) at (0,1.15)  [mutable] {};
	\node (2) at (-1,-.58)  [mutable] {};
	\node (3) at (1,-.58)  [mutable] {};
	\draw (1) to [out=255, in=45,-angle 90] (2);
	\draw (2) to [out=15,in=165,-angle 90] (3);
	\draw (3) to [out=135,in=285,-angle 90] (1);
	\draw (1) to [out=210, in=90,-angle 90] (2);
	\draw (2) to [out=330,in=210,-angle 90] (3);
	\draw (3) to [out=90,in=330,-angle 90] (1);
\end{tikzpicture}
\caption{The Markov quiver}
\label{fig: Markov}
\end{figure}

The Markov cluster algebra is famous for its many pathologies, and for its usefulness in disproving general facts that might be true about cluster algebras.  It was introduced in \cite{BFZ05} as an example of a cluster algebra distinct from its own upper cluster algebra.  The Markov cluster algebra is the smallest example of a non-locally-acyclic cluster algebra, which will follow from any of the following propositions.

The quiver $\Q$ has the property that every mutation of $\Q$ is itself.  Moreso is true; for any pair of cluster variables $a,b$, there is an automorphism of $\A$ which sends cluster variables to cluster variables, clusters to clusters, and $a$ to $b$.

Furthermore, every vertex of $\Q$ has two incoming arrows, and two outgoing arrows.  Because of this, every cluster relation in $\A$ is homogenous of degree 2 in the cluster variables.  Therefore, $\A$ is an $\mathbb{N}$-graded algebra, with the degree of the cluster variables equal to $1$.

The Markov cluster algebra arises as the cluster algebra of a marked surface.
\begin{prop}
Let $(\S,\M)$ be the torus with no boundary, and a single puncture.  Then $\A=\A(\S,\M)$.
\end{prop}
\begin{proof}
Triangulate $(\S,\M)$ with three arcs as follows, where opposite sides of this square have been identified.
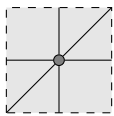
\begin{center}
\begin{tikzpicture}[scale=.7]
    \draw [dashed,fill=black!10] (-1,-1) to (-1,1) to (1,1) to (1,-1) to (-1,-1);
    \node (1) at (0,0) [marked] {};
    \draw (1) to (0,1);
    \draw (1) to (1,1);
    \draw (1) to (0,-1);
    \draw (1) to (-1,-1);
    \draw (1) to (1,0);
    \draw (1) to (-1,0);
\end{tikzpicture}
\end{center}
This triangulation gives the quiver $\Q$.
\end{proof}
\begin{rem}
All triangulations of this surface are diffeomorphic; this gives an explanation why every mutation of $\Q$ is equivalent to itself.
\end{rem}

By Theorem \ref{thm: noboundary}, $\A$ is \textbf{not locally acyclic}.  We now show that $\A$ lacks many of the nice properties of locally acyclic cluster algebras.

\begin{prop}
$\A$ is infinitely generated and non-Noetherian.
\end{prop}
\begin{proof}
Since $\Q$ is not finite-type, there are an infinite number of cluster variables.  Therefore, $\A_1$ is an infinitely-generated module over $\A_0=\Z$, and so $\A$ cannot be finitely generated.  Now, let $\{S_i\}$, $i\in\mathbb{N}$ be collections of cluster variables such that $i<j$ implies that $S_i\subset S_j$.  Then $\{\A S_i\}$ define an ascending sequence of ideals which do not stabilize.
\end{proof}
\begin{rem}
The prime ideal $\A_{\geq1}$ has infinite Zariski tangent dimension, so $\A$ is infinitely generated even to first order around that point in $\X_\A$.
\end{rem}

\begin{prop}\cite[Proposition 1.26]{BFZ05}
The Markov cluster algebra $\A\neq \U$.
\end{prop}
\begin{proof}
For $(x,y,z)$ any cluster, the element
$ m=\frac{x^2+y^2+z^2}{xyz}$
is Laurent in each cluster, so it is in $\U$.  However, it has graded degree $-1$, so it is not in $\A$.
\end{proof}

\subsection{The torus with one boundary marked point} Let $\Q$ be the quiver in Figure \ref{fig: torus}, and $\A=\A(\Q)$.
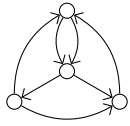
\begin{figure}[h!]
\begin{tikzpicture}[scale=.7,auto,minimum size=2mm]
    \node (1) at (-1,-.58) [mutable] {};
    \node (2) at (1,-.58) [mutable] {};
    \node (3) at (0,1.15) [mutable] {};
    \node (4) at (0,0) [mutable] {};

    \draw (1) to [out=330,in=210,-angle 90] (2);
    \draw (2) to [out=90,in=330,-angle 90] (3);
    \draw (1) to [out=90,in=210,-angle 90] (3);
    \draw (4) to [-angle 90] (1);
    \draw (4) to [-angle 90] (2);
    \draw (3) to [out=300,in=60,-angle 90] (4);
    \draw (3) to [out=240,in=120,-angle 90] (4);
\end{tikzpicture}
\caption{The torus with one boundary marked point}
\label{fig: torus}
\end{figure}

This quiver is closely related to the Markov quiver.  Every mutation of $\Q$ is itself, and for any pair of cluster variables $a,b$, there is an automorphism of $\A$ which sends cluster variables to cluster variables, clusters to clusters, and $a$ to $b$.

\begin{prop}
Let $(\S,\M)$ be the torus with a disc cut out, and one marked point on the boundary.  Then $\A=\A(\S,\M)$.
\end{prop}
\begin{proof}
Triangulate $(\S,\M)$ with four arcs as follows, where the two dashed circles have been identified (with opposite orientation).
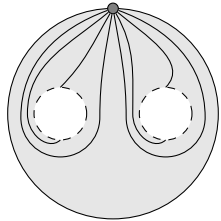
\begin{center}
\begin{tikzpicture}[scale=.35]
    \draw[fill=black!10] (0,0) circle (4);
    \node (1) at (0,4) [marked] {};
    \node (2) at (3.5,-.5) {};
    \node (3) at (.5,-.5) {};
    \node (4) at (-3.5,-.5) {};
    \node (5) at (-.5,-.5) {};
    \draw (1) to [out=-20,in=90] (2) arc (360:180:1.5) (3) to [out=90,in=280] (1);
    \draw (1) to [out=200,in=90] (4) arc (180:360:1.5) (5) to [out=90,in=260] (1);
    \draw[fill=white,dashed] (2,0) circle (1);
    \draw[fill=white,dashed] (-2,0) circle (1);
    \draw (1) to [out=-40,in=45] (2,1);
    \draw (1) to [out=300,in=90] (.75,0) arc (180:225:1.25) to [out=-45,in=225] (2,-1);
    \draw (1) to [out=240,in=45] (-2,1);
    \draw (1) to [out=220,in=90] (-3.25,0) arc (180:225:1.25) to [out=-45,in=225] (-2,-1);
\end{tikzpicture}
\end{center}
This triangulation gives the quiver $\Q$.
\end{proof}
\begin{rem}
All triangulations of this surface are diffeomorphic; this gives an explanation why every mutation of $\Q$ is equivalent to itself.
\end{rem}

By Theorem \ref{thm: oneboundary}, $\A$ is \textbf{not locally acyclic}.  Much like the Markov cluster algebra, it exhibits several pathological behaviors; we focus on where it differs from the Markov cluster algebra.

Every vertex in every seed has both incoming and outgoing arrows.  This implies that the ideal $P$ generated by the cluster variables is non-trivial.
\begin{prop}$P^2=P$.
\end{prop}
\begin{proof}
Consider a generic seed of $\A$ given below.
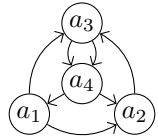
\begin{center}
\begin{tikzpicture}[scale=.7,auto,minimum size=2mm]
    \node (1) at (-1,-.58) [mutable] {$a_1$};
    \node (2) at (1,-.58) [mutable] {$a_2$};
    \node (3) at (0,1.15) [mutable] {$a_3$};
    \node (4) at (0,0) [mutable] {$a_4$};

    \draw (1) to [out=330,in=210,-angle 90] (2);
    \draw (2) to [out=90,in=330,-angle 90] (3);
    \draw (1) to [out=90,in=210,-angle 90] (3);
    \draw (4) to [-angle 90] (1);
    \draw (4) to [-angle 90] (2);
    \draw (3) to [out=300,in=60,-angle 90] (4);
    \draw (3) to [out=240,in=120,-angle 90] (4);
\end{tikzpicture}
\end{center}
The mutation relation at $a_1$ is
\[ a_1a_1' = a_2a_3+a_4\]
and so $a_4\in P^2$.  Because there is a cluster automorphism of $\A$ which takes $a_4$ to any other cluster variable, they are all in $P$, and so $P^2=P$.
\end{proof}
\begin{rem}
The prime ideal $P$ in $\X_\A$ has zero Zariski tangent dimension.  However, it is not an isolated point, because $\A$ is a domain and so $\X_\A$ is irreducible.  The interpretation is that $P$ is a singularity in $\X_\A$ so twisted, there are no first-order deformations of $P$.  This is a necessarily infinite-dimensional phenomenon.
\end{rem}
\begin{coro}\label{coro: torusnotfg}
$\A$ is not finitely-generated.
\end{coro}
\begin{proof}
If $\A$ is finitely generated, then $P$ is also finitely-generated.  Then  $P\cdot P=P$ and Nakayama's lemma imply that there is a non-zero element in $\A$ which kills $P$.  Since $\A$ is a domain, this is a contradiction.
\end{proof}


\subsection{The torus with two boundary marked points}
Let $\Q$ be the quiver in Figure \ref{fig: torus2}, and $\A=\A(\Q)$.
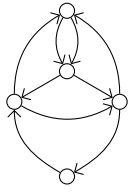
\begin{figure}[h!]
\centering
\begin{tikzpicture}[scale=.7,auto,minimum size=2mm]
    \node (1) at (-1,-.58) [mutable] {};
    \node (2) at (1,-.58) [mutable] {};
    \node (3) at (0,1.15) [mutable] {};
    \node (4) at (0,0) [mutable] {};
    \node (5) at (0,-2) [mutable] {};

    \draw (1) to [out=330,in=210,-angle 90] (2);
    \draw (2) to [out=90,in=330,-angle 90] (3);
    \draw (1) to [out=90,in=210,-angle 90] (3);
    \draw (4) to [-angle 90] (1);
    \draw (4) to [-angle 90] (2);
    \draw (3) to [out=300,in=60,-angle 90] (4);
    \draw (3) to [out=240,in=120,-angle 90] (4);
    \draw (1) to [out=270,in=150,angle 90-] (5);
    \draw (2) to [out=270,in=30,-angle 90] (5);
\end{tikzpicture}
\caption{Torus with two boundary marked points}
\label{fig: torus2}
\end{figure}

\begin{prop}
Let $(\S,\M)$ be the torus with a disc cut out, and two marked points on the boundary.  Then $\A=\A(\S,\M)$.
\end{prop}
\begin{proof}
Triangulate $(\S,\M)$ with five arcs as follows, where the two dashed circles have been identified (with opposite orientation).
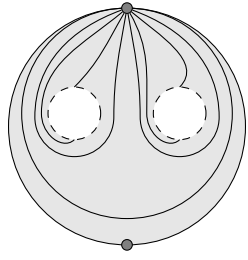
\begin{center}
\begin{tikzpicture}[scale=.35]
    \draw[fill=black!10] (0,-.5) circle (4.5);
    \draw (0,0) circle (4);
    \node (1) at (0,4) [marked] {};
    \node (6) at (0,-5) [marked] {};
    \node (2) at (3.5,-.5) {};
    \node (3) at (.5,-.5) {};
    \node (4) at (-3.5,-.5) {};
    \node (5) at (-.5,-.5) {};
    \draw (1) to [out=-20,in=90] (2) arc (360:180:1.5) (3) to [out=90,in=280] (1);
    \draw (1) to [out=200,in=90] (4) arc (180:360:1.5) (5) to [out=90,in=260] (1);
    \draw[fill=white,dashed] (2,0) circle (1);
    \draw[fill=white,dashed] (-2,0) circle (1);
    \draw (1) to [out=-40,in=45] (2,1);
    \draw (1) to [out=300,in=90] (.75,0) arc (180:225:1.25) to [out=-45,in=225] (2,-1);
    \draw (1) to [out=240,in=45] (-2,1);
    \draw (1) to [out=220,in=90] (-3.25,0) arc (180:225:1.25) to [out=-45,in=225] (-2,-1);
\end{tikzpicture}
\end{center}
This triangulation gives the quiver $\Q$.
\end{proof}

By Theorem \ref{thm: atleast2}, $\A$ is \textbf{locally acyclic} and all that this implies. In particular, $\A$ is finitely generated.

Let $i$ be the vertex in $\Q$ with only two incident arrows (the bottom-most one, as drawn).  Let $\Q'$ be the freezing of $\Q$ at $i$, and let $\Q^\dag$ be the deletion of $i$ in $\Q$.
\begin{prop}\label{prop: notclusterlocal}
The natural inclusions
\[ \A(\Q')\subset \A[a_i^{-1}],\;\;\; \A(\Q^\dag)\subset \A/\langle a_i-1\rangle\]
are not equalities.  In particular, $\A(\Q')$ is a freezing which is not a cluster localization.
\end{prop}
\begin{proof}
The quiver $\Q^\dag$ is the quiver in Figure \ref{fig: torus}, so by Corollary \ref{coro: torusnotfg}, $\A(\Q^\dag)$ is not finitely-generated.  Since $\A$ is finitely generated, $\A(\Q^\dag)\neq \A/\langle a_i-1\rangle$.
If $\A(\Q')=\A[a_i^{-1}]$, then quotienting both sides by $a_i-1$ would give $\A(\Q^\dag)=\A/\langle a_i-1\rangle$, so this is impossible.
\end{proof}
\begin{coro}
$\A(\Q')\neq\U(\Q')$
\end{coro}
\begin{proof}
Since $\A(\Q')$ is a freezing which is not a cluster localization, Lemma \ref{lemma: acycliclocal} implies that $\A(\Q')\neq \U(\Q')$.
\end{proof}
\begin{rem}\label{rem: Banfffail}
This is an example of a cluster algebra for which the Banff algorithm \emph{can} fail, but does not \emph{always} fail.
\end{rem}

\subsection{Sphere with four punctures}
Let $\Q$ be the quiver in Figure \ref{fig: sphere4}, and $\A=\A(\Q)$.

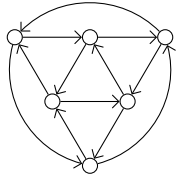
\begin{figure}[h!]
\centering
\begin{tikzpicture}[scale=.5,auto,minimum size=2mm]
    \node (1) at (0,1.14) [mutable] {};
    \node (2) at (-1,-.57) [mutable] {};
    \node (3) at (1,-.57) [mutable] {};
    \node (4) at (-2,1.14) [mutable] {};
    \node (5) at (2,1.14) [mutable] {};
    \node (6) at (0,-2.28) [mutable] {};

    \draw (1) to [-angle 90] (2);
    \draw (2) to [-angle 90] (3);
    \draw (3) to [-angle 90] (1);
    \draw (4) to [-angle 90] (1);
    \draw (1) to [-angle 90] (5);
    \draw (5) to [-angle 90] (3);
    \draw (3) to [-angle 90] (6);
    \draw (6) to [-angle 90] (2);
    \draw (2) to [-angle 90] (4);
    \draw (4) to [angle 90-,out=45,in=135] (5);
    \draw (5) to [angle 90-,out=285,in=15] (6);
    \draw (6) to [angle 90-,out=165,in=255] (4);
\end{tikzpicture}
\caption{Sphere with four punctures}
\label{fig: sphere4}
\end{figure}

\begin{prop}
Let $(\S,\M)$ be the sphere with four punctures.  Then $\A=\A(\S,\M)$.
\end{prop}
\begin{proof}
Triangulate $(\S,\M)$ with six arcs as follows, where the dashed circle denotes a single point at infinity which has been blown up to allow for a planar embedding.
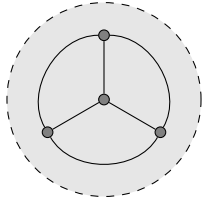
\begin{center}
\begin{tikzpicture}[scale=.75]
    \draw[fill=black!10,dashed] (0,0) circle (1.72);
    \node (1) at (-1,-.58) [marked] {};
    \node (2) at (1,-.58) [marked] {};
    \node (3) at (0,1.14) [marked] {};
    \node (4) at (0,0) [marked] {};
    \draw (1) to [out=-60,in=240] (2);
    \draw (2) to [out=60,in=0] (3);
    \draw (3) to [out=180,in=120] (1);
    \draw (4) to  (1);
    \draw (4) to  (2);
    \draw (3) to (4);
\end{tikzpicture}
\end{center}
This triangulation gives the quiver $\Q$.
\end{proof}
\noindent By Theorem \ref{thm: noboundary}, $\A$ is \textbf{not locally acyclic}.

For a given puncture $p$ in $(\S,\M)$, there is an involution of $\A$ which changes each plain tagging at $p$ to notched, and each notched tagging at $p$ to plain.  The four punctures of $(\S,\M)$ then give an action of the group $\Gamma=(\Z_2)^4$ on $\A$.

Let $\Delta$ be the triangulation of $(\S,\M)$ in the proof of the proposition, and let $\Gamma(\Delta)$ be the collection of 24 tagged arcs in the orbit of $\Delta$ under the group $\Gamma$.
\begin{prop}\label{prop: spheregenerate}
The set of cluster variables $\Gamma(\Delta)$ generates $\A$.
\end{prop}
The key lemma in the proof will be the following.
\begin{lemma}
Let $\alpha,\beta,\gamma,\delta$ be tagged arcs in $(\S,\M)$ be as pictured below (the endpoints of $\alpha$ need not be on the boundary or distinct).
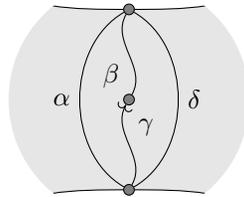
\begin{center}
\begin{tikzpicture}
\begin{scope}[scale=.4,auto]
\begin{scope}
	\clip (0,0) circle (4);
	\draw[fill=black!10] (-5,-4) to [in=180,out=30] (0,-3) to [in=150,out=0] (5,-4) to [line to] (5,4) to [in=0,out=210] (0,3) to [in=-30,out=180] (-5,4);
	\node (1) at (0,-3) [marked] {};
	\node (2) at (0,0) [marked] {};
	\node (3) at (0,3) [marked] {};
	\draw (1) to [in=240,out=60,-hooks reversed] node [swap] {$\gamma$} (2);
	\draw (2) to [in=240,out=60] node  {$\beta$} (3);
	\draw (1) to [in=210,out=150] node {$\alpha$} (3);
	\draw (1) to [in=330,out=30] node [swap] {$\delta$} (3);
\end{scope}
\end{scope}
\end{tikzpicture}
\end{center}
Then $\alpha=\beta\gamma-\delta$ in $\A(\S,\M)$.
\end{lemma}
\begin{proof}
Let $\beta'$ be $\beta$ with a notch at the puncture.  Then $\alpha,\beta,\beta',\delta$ can be extended to a tagged triangulation $\Delta$ of $(\S,\M)$. In this triangulation, $\gamma$ is the mutation of $\beta$, and the corresponding cluster relation is $\beta\gamma=\alpha+\delta$.
\end{proof}

\begin{proof}[Proof of Proposition \ref{prop: spheregenerate}]
Let $R\subset \A$ be the subalgebra generated by $\Gamma(\Delta)$.  Since the generators of $R$ are closed under $\Gamma$, $R$ is closed under $\Gamma$.  Therefore, it will suffice to show that $R$ contains every plain tagged arc in $\A$.

For any plain tagged arc $\alpha$, let $i$ be its intersection number with $\Delta$. If $i=0$, then $\alpha\in \Delta$, and $\alpha\in R$.  If $i=1$, then there are $\beta,\gamma',\delta\in \Gamma(\Delta)$ as in the picture below.
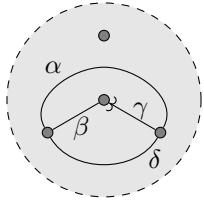
\begin{center}
\begin{tikzpicture}[scale=.75,auto]
    \draw[fill=black!10,dashed] (0,0) circle (1.72);
    \node (1) at (-1,-.58) [marked] {};
    \node (2) at (1,-.58) [marked] {};
    \node (3) at (0,1.14) [marked] {};
    \node (4) at (0,0) [marked] {};
    \draw (1) to [out=-60,in=240]  (2);
    \draw (4) to (1);
    \draw (4) to [hooks reversed-] (2);
    \draw (1) to [out=120,in=180] (0,.58) to [out=0,in=60] (2);
    \node (a) at (-.9,.58) {$\alpha$};
    \node (b) at (-.4,-.5) {$\beta$};
    \node (c) at (.65,-.2) {$\gamma$};
    \node (d) at (.9,-1.05) {$\delta$};
\end{tikzpicture}
\end{center}
Then $\alpha=\beta\gamma-\delta\in R$.

Now, assume that every tagged arc which intersects $\Delta$ fewer than $i$ times is in $R$, and let $\alpha$ be a plain tagged arc which intersects $\Delta$ $i$ times, with $i\geq2$.  Since $\alpha$ intersects at least two arcs in $\Delta$, there will be a puncture in $(\S,\M)$ around which $\alpha$ looks like the figure on the left (where the dashed lines represent arcs in $\Delta$, and the picture only denotes a small disc around that puncture).
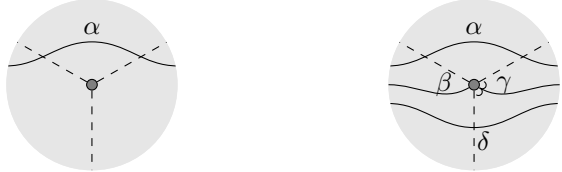
\begin{center}
\begin{tikzpicture}[scale=.5]
\begin{scope}[xshift=-2in]
	\clip (0,0) circle (2.28);
	\draw[fill=black!10,draw=black!10] (0,0) circle (2.28);
	\node (1) at (0,0) [marked] {};
	\draw (-2.28,0.5) to [out=0,in=180] (0,1.14) to [out=0,in=180] (2.28,0.5);
	\draw (1) to [dashed] (0,-2.28);
	\draw (1) to [dashed] (2,1.14);
	\draw (1) to [dashed] (-2,1.14);
	\node (a) at (0,1.5) {$\alpha$};
\end{scope}
\begin{scope}[xshift=2in]
	\clip (0,0) circle (2.28);
	\draw[fill=black!10,draw=black!10] (0,0) circle (2.28);
	\node (1) at (0,0) [marked] {};
	\draw (-2.28,-0.5) to [out=0,in=180] (0,-1.14) to [out=0,in=180] (2.28,-0.5);
	\draw (-2.28,0.5) to [out=0,in=180] (0,1.14) to [out=0,in=180] (2.28,0.5);
	\draw (-2.28,0) to [out=0,in=210] (1);
	\draw (2.28,0) to [out=180,in=-30,-hooks reversed] (1);
	\draw (1) to [dashed] (0,-2.28);
	\draw (1) to [dashed] (2,1.14);
	\draw (1) to [dashed] (-2,1.14);
	\node (a) at (0,1.5) {$\alpha$};
	\node (b) at (-.8,0) {$\beta$};
	\node (c) at (.8,0) {$\gamma$};
	\node (d) at (0.25,-1.5) {$\delta$};
\end{scope}
\end{tikzpicture}
\end{center}
Let $\beta,\gamma,\delta$ be as in the figure on the right.  Each of $\beta,\gamma$ and $\delta$ intersect $\Delta$ fewer than $i$ times, so each are in $R$.  Then $\alpha=\beta\gamma-\delta\in R$.  By induction, every plain tagged arc is in $R$, and by $\Gamma$-invariance, every tagged arc is in $R$, and so $\A=R$.
\end{proof}
Then, $\A$ is an example of a cluster algebra which is finitely generated, but is not locally acyclic.

\subsection{A locally acyclic, non-surface cluster algebra}
Let $\Q$ be the quiver in Figure \ref{fig: X6}, and $\A=\A(\Q)$.
\begin{figure}[h!]
\centering
\begin{tikzpicture}[inner sep=0.5mm,scale=.5,auto,minimum size=2mm]
	\node (1) at (0,0)  [mutable] {};
	\node (2) at (-2,0)  [mutable] {};
	\node (3) at (2,0)  [mutable] {};
	\node (4) at (-1,1.5)  [mutable] {};
	\node (5) at (1,1.5)  [mutable] {};
	\node (6) at (0,-1.3)  [mutable] {};
	\draw (1) to [-angle 90] (2);
	\draw (2) to [out=80, in=220,-angle 90] (4);
	\draw (2) to [out=40, in=260,-angle 90] (4);
	\draw (4) to [-angle 90] (1);
	\draw (1) to [-angle 90] (3);
	\draw (3) to [out=140, in=280,-angle 90] (5);
	\draw (3) to [out=100, in=320,-angle 90] (5);
	\draw (5) to [-angle 90] (1);
	\draw (1) to [-angle 90] (6);
\end{tikzpicture}
\caption{$X_6$}
\label{fig: X6}
\end{figure}

This quiver (called $X_6$) was introduced by Derksen and Owen in \cite{DO08} as an example of a mutation-finite quiver which does not come from a surface.
\begin{prop}
$\A$ is not acyclic, but is locally acyclic.
\end{prop}
\begin{proof}
The five quivers mutation-equivalent to $\Q$ can be found in \cite[Proposition 4]{DO08}; each contains a cycle.  Local acyclicity is shown by the Banff algorithm in Figure \ref{fig: Banff2}.
\end{proof}
\begin{rem}
The related $X_7$ quiver of Derksen and Owens (Figure \ref{fig: X7}) has no covering pairs and is not locally acyclic.  This can be verified by checking the two possible quivers obtained by mutation have no sinks or sources.
\begin{figure}[h!]
\centering
\begin{tikzpicture}[inner sep=0.5mm,scale=.5,auto,minimum size=2mm]
	\node (1) at (0,0)  [mutable] {};
	\node (2) at (-2,0)  [mutable] {};
	\node (3) at (2,0)  [mutable] {};
	\node (4) at (-1,1.5)  [mutable] {};
	\node (5) at (1,1.5)  [mutable] {};
	\node (6) at (-1,-1.5)  [mutable] {};
	\node (7) at (1,-1.5)  [mutable] {};
	\draw (1) to [-angle 90] (2);
	\draw (2) to [out=80, in=220,-angle 90] (4);
	\draw (2) to [out=40, in=260,-angle 90] (4);
	\draw (4) to [-angle 90] (1);
	\draw (1) to [angle 90-] (3);
	\draw (3) to [out=140, in=280,angle 90-] (5);
	\draw (3) to [out=100, in=320,angle 90-] (5);
	\draw (5) to [angle 90-] (1);
	\draw (1) to [-angle 90] (7);
	\draw (7) to [out=160, in=20,-angle 90] (6);
	\draw (7) to [out=200, in=-20,-angle 90] (6);
	\draw (6) to [-angle 90] (1);
\end{tikzpicture}
\caption{$X_7$}
\label{fig: X7}
\end{figure}
\end{rem}

\bibliography{MyNewBib}{}

\def\cprime{$'$} \def\cprime{$'$} \def\cprime{$'$} \def\cprime{$'$}
  \def\cprime{$'$}
\providecommand{\bysame}{\leavevmode\hbox to3em{\hrulefill}\thinspace}
\providecommand{\MR}{\relax\ifhmode\unskip\space\fi MR }
\providecommand{\MRhref}[2]{%
  \href{http://www.ams.org/mathscinet-getitem?mr=#1}{#2}
}
\providecommand{\href}[2]{#2}
\begin{thebibliography}{GSV10}

\bibitem[BFZ05]{BFZ05}
Arkady Berenstein, Sergey Fomin, and Andrei Zelevinsky, \emph{Cluster algebras.
  {III}. {U}pper bounds and double {B}ruhat cells}, Duke Math. J. \textbf{126}
  (2005), no.~1, 1--52. \MR{2110627 (2005i:16065)}

\bibitem[DO08]{DO08}
Harm Derksen and Theodore Owen, \emph{New graphs of finite mutation type},
  Electron. J. Combin. \textbf{15} (2008), no.~1, Research Paper 139, 15.
  \MR{2465763 (2009j:05102)}

\bibitem[Eis95]{Eis95}
David Eisenbud, \emph{Commutative algebra}, Graduate Texts in Mathematics, vol.
  150, Springer-Verlag, New York, 1995, With a view toward algebraic geometry.
  \MR{MR1322960 (97a:13001)}

\bibitem[FG06]{FG06}
Vladimir Fock and Alexander Goncharov, \emph{Moduli spaces of local systems and
  higher {T}eichm\"uller theory}, Publ. Math. Inst. Hautes \'Etudes Sci.
  (2006), no.~103, 1--211. \MR{2233852 (2009k:32011)}

\bibitem[FST]{FST??}
Sergey Fomin, Michael Shapiro, and Dylan Thurston, \emph{Cluster algebras and
  triangulated surfaces. {II}. {C}luster complexes}, to appear.

\bibitem[FST08]{FST08}
\bysame, \emph{Cluster algebras and triangulated surfaces. {I}. {C}luster
  complexes}, Acta Math. \textbf{201} (2008), no.~1, 83--146. \MR{2448067
  (2010b:57032)}

\bibitem[FZ99]{FZ99}
Sergey Fomin and Andrei Zelevinsky, \emph{Double {B}ruhat cells and total
  positivity}, J. Amer. Math. Soc. \textbf{12} (1999), no.~2, 335--380.
  \MR{1652878 (2001f:20097)}

\bibitem[FZ02]{FZ02}
\bysame, \emph{Cluster algebras. {I}. {F}oundations}, J. Amer. Math. Soc.
  \textbf{15} (2002), no.~2, 497--529 (electronic). \MR{1887642 (2003f:16050)}

\bibitem[FZ03]{FZ03}
\bysame, \emph{Cluster algebras. {II}. {F}inite type classification}, Invent.
  Math. \textbf{154} (2003), no.~1, 63--121. \MR{2004457 (2004m:17011)}

\bibitem[GLS08]{GLS08}
Christof Geiss, Bernard Leclerc, and Jan Schr{\"o}er, \emph{Partial flag
  varieties and preprojective algebras}, Ann. Inst. Fourier (Grenoble)
  \textbf{58} (2008), no.~3, 825--876. \MR{2427512 (2009f:14104)}

\bibitem[GSV03]{GSV03}
Michael Gekhtman, Michael Shapiro, and Alek Vainshtein, \emph{Cluster algebras
  and {P}oisson geometry}, Mosc. Math. J. \textbf{3} (2003), no.~3, 899--934,
  1199, \{Dedicated to Vladimir Igorevich Arnold on the occasion of his 65th
  birthday\}. \MR{2078567 (2005i:53104)}

\bibitem[GSV05]{GSV05}
\bysame, \emph{Cluster algebras and {W}eil-{P}etersson forms}, Duke Math. J.
  \textbf{127} (2005), no.~2, 291--311. \MR{2130414 (2006d:53103)}

\bibitem[GSV10]{GSV10}
\bysame, \emph{Cluster algebras and {P}oisson geometry}, Mathematical Surveys
  and Monographs, vol. 167, American Mathematical Society, Providence, RI,
  2010. \MR{2683456}

\bibitem[Har77]{Har77}
Robin Hartshorne, \emph{Algebraic geometry}, Springer-Verlag, New York, 1977,
  Graduate Texts in Mathematics, No. 52. \MR{MR0463157 (57 \#3116)}

\bibitem[Kel]{QuiverApplet}
B~Keller, \emph{Quiver mutation in java}.

\bibitem[Mul]{GregWP}
Greg Muller, \emph{The {W}eil {P}etersson form on an acyclic cluster variety.},
  arxiv: 1103.2341.

\end{thebibliography}
\bibliographystyle{amsalpha}
\end{document}